\documentclass[11pt]{amsart}
\usepackage[T1]{fontenc}
\usepackage[utf8]{inputenc}
\usepackage{lmodern}
\usepackage{amssymb,mathtools}
\usepackage{booktabs}
\usepackage{microtype}
\AtBeginDocument{{\footnotesize\fontencoding{OT1}\selectfont}}
\usepackage{needspace}
\usepackage[margin=2.5cm]{geometry}
\usepackage{hyperref}
\hypersetup{hidelinks,unicode=true,
  pdftitle={On union-closed families with prescribed number of k-sets},
  pdfauthor={Amir Jafari},
  pdfsubject={Extremal set theory; layered Leck-Roberts-Simpson conjecture},
  pdfkeywords={union-closed families, extremal set theory, max-lex order, binomial strips}}
\theoremstyle{plain}
\newtheorem{theorem}{Theorem}[section]
\newtheorem{conjecture}[theorem]{Conjecture}
\newtheorem*{weightedconjecture}{Conjecture}
\newtheorem{lemma}[theorem]{Lemma}
\newtheorem{proposition}[theorem]{Proposition}
\newtheorem{corollary}[theorem]{Corollary}
\theoremstyle{definition}
\newtheorem{definition}[theorem]{Definition}
\theoremstyle{remark}
\newtheorem*{remark}{Remark}

\numberwithin{equation}{section}
\title{On union-closed families with prescribed number of $k$-sets}

\author[A.~Jafari]{Amir Jafari}
\address{Department of Mathematics, New Uzbekistan University,
Tashkent 100007, Uzbekistan}
\email{a.jafari@newuu.uz}

\subjclass[2020]{Primary 05D05}
\keywords{Union-closed families, extremal set theory, max-lex order,
binomial strips}
\date{}

\begin{document}
\begin{abstract}
Fix positive integers \(N,k,n\) with \(n\ge k\). We seek the minimum number of members of size at least \(n\) in a finite family of finite sets closed under union and containing exactly \(N\) distinct sets of size \(k\). This problem is a specialization of the Leck--Roberts--Simpson weighted conjecture: assign weight one to sets of size at least \(n\) and zero to smaller sets. The predicted minimizer consists of the unions of nonempty subfamilies of the first \(N\) \(k\)-subsets of the natural numbers, ordered by their largest elements and, when these agree, by their increasing lists lexicographically. For an integer \(t\ge1\), call the range
\[
 \binom{n+t-1}{k}<N\le\binom{n+t}{k}
\]
the \(t\)-th strip.

We prove the layered conjecture throughout the first strip, and throughout the second strip for \(k=3\). For arbitrary \(k\), we prove the second strip for families of subsets of an \((n+2)\)-element set. For \(k,t\ge3\), we prove the \(t\)-th strip for families of subsets of an \((n+t)\)-element set whenever \(n\ge(t+1)(k-1)\). With no restriction on the ground set, we prove it for \(k\ge3\) and \(t\ge2\) whenever \(n>\frac52 k^2t\). For sufficiently large \(k\), we obtain a sufficient bound of order \(k^2t/\log k\), uniformly in \(t\ge2\).
\end{abstract}
\maketitle

\hypertarget{introduction}{%
\section{Introduction}\label{introduction}}

A family \(\mathcal F\) of sets is called \emph{union-closed} if \(A\cup B\in\mathcal F\) whenever \(A,B\in\mathcal F\). All sets and families in this paper are finite. Fix positive integers \(N,k,n\) with \(n\ge k\). Among union-closed families containing exactly \(N\) sets of size \(k\), we seek the minimum number of members of size at least \(n\). A set of size \(k\) will be called a \emph{\(k\)-set}.

For any family \(\mathcal A\), let \(\langle\mathcal A\rangle\) denote the family of all unions of nonempty subfamilies of \(\mathcal A\). The empty union is not adjoined; in particular, \(\langle\varnothing\rangle=\varnothing\). This is the smallest union-closed family containing \(\mathcal A\); the members of \(\mathcal A\) will be called its \emph{generators}. Every union-closed family contains the union-closure of its \(k\)-sets, and taking unions of \(k\)-sets adds no new \(k\)-sets. Thus it is equivalent to minimize the number of large unions generated by \(N\) distinct \(k\)-sets. When \(n=k\), every generated member meets the cutoff, so the problem reduces to minimizing \(|\langle\mathcal A\rangle|\).

The problem of minimizing the size of a union-closure was considered by Roberts in his 1999 thesis \cite{Roberts1999}. The conjecture proposed there was proved for \(k=2\) by Leck, Roberts, and Simpson \cite{LeckRobertsSimpson2012}. They also proposed the weighted version stated below.

More recently, Ranđelović \cite{Randjelovic2026} proved Roberts's size conjecture when \(N=\binom{m}{k}\), for each fixed \(k\) and all sufficiently large \(m\). In this case the minimum is attained by taking all \(k\)-subsets of an \(m\)-element set. He also gave another proof for \(k=2\) and established further cases with \(N\) close to a binomial coefficient.

To state the conjecture, we first recall the max-lex order on the \(k\)-subsets of the natural numbers. We write \(\mathbb N=\{1,2,3,\ldots\}\).

\begin{definition}\label{def:maxlex} Let \(k\ge1\). For distinct \(k\)-subsets \(A\) and \(B\) of \(\mathbb N\), we say that \(A\) precedes \(B\) in \emph{max-lex order} if \(\max A<\max B\), or if \(\max A=\max B\) and the smallest integer belonging to exactly one of \(A\) and \(B\) belongs to \(A\).
\end{definition}

Thus we compare the largest elements first. When these agree, we compare the increasing lists of elements lexicographically. For example, \(\{1,4,5\}\) precedes \(\{2,3,5\}\).

For positive integers \(N\) and \(k\), let \(\mathcal F_k(N)\) denote the family of the first \(N\) \(k\)-sets in this order. We write \([m]=\{1,\ldots,m\}\), with \([0]=\varnothing\), and \(\binom Vj\) for the family of all \(j\)-subsets of a finite set \(V\). For integer arguments, \(\binom ab=0\) unless \(0\le b\le a\).

Given nonnegative weights \(0\le w_0\le w_1\le\cdots\), define the weight of a finite family \(\mathcal F\) by
\[
 w(\mathcal F)=\sum_{A\in\mathcal F}w_{|A|}.
\]
Leck, Roberts, and Simpson \cite[Conjecture~6]{LeckRobertsSimpson2012} proposed the following bound.

\begin{weightedconjecture}[Leck--Roberts--Simpson]
Let \(N,k\ge1\). For every family \(\mathcal A\) of \(N\) distinct \(k\)-sets and every nonnegative nondecreasing sequence of weights,
\[
 w(\langle\mathcal A\rangle)
 \ge w(\langle\mathcal F_k(N)\rangle).
\]
\end{weightedconjecture}

The choice \(w_j=1\) for every \(j\) gives the size conjecture. We shall use weights that count only members above a prescribed size cutoff.

For a family \(\mathcal F\) and integers \(j,n\), put \[
\mathcal F_j=\{A\in\mathcal F:|A|=j\},
\qquad
\mathcal F_{\ge n}=\{A\in\mathcal F:|A|\ge n\}.
\] We call \(\mathcal F_j\) the \(j\)-th \emph{level}, or \emph{layer}, and the members of \(\mathcal F_{\ge n}\) the \emph{large members}, with the cutoff \(n\) understood. A subscript \(\le n\) has the analogous meaning. These conventions also apply at nonpositive cutoffs; thus \(\mathcal F_{\ge0}=\mathcal F\). Write
\[
 U_k(N,n)=|\langle\mathcal F_k(N)\rangle_{\ge n}|
\]
for the number supplied by the max-lex family. Taking \(w_j=0\) for \(j<n\) and \(w_j=1\) for \(j\ge n\) gives our layered formulation as a special case. It applies to every union-closed family containing the generators, since such a family contains their union-closure.

\begin{conjecture}[Layered form]\label{conj:layered} Let \(N,k,n\) be positive integers with \(n\ge k\). If \(\mathcal F\) is a union-closed family containing exactly \(N\) \(k\)-sets, then \[
|\mathcal F_{\ge n}|
\ge
|\langle\mathcal F_k(N)\rangle_{\ge n}|.
\]
\end{conjecture}

Applied to generated families, the case \(n=k\) gives Roberts's size conjecture. On the finite range of sizes occurring in the two closures, every nonnegative nondecreasing weight is a nonnegative linear combination of threshold weights. Thus our layered formulation, taken over all cutoffs, is equivalent to the weighted conjecture.

\begin{definition}\label{def:support} The \emph{support} of a family \(\mathcal F\) is the set \[
\operatorname{supp}(\mathcal F)=\bigcup_{A\in\mathcal F}A.
\] Its cardinality is called the \emph{support-size} of \(\mathcal F\). For positive integers \(N\) and \(k\), let \(\sigma_k(N)\) denote the least possible support-size of a family of \(N\) distinct \(k\)-sets. Thus \[
\sigma_k(N)=\min\left\{m\in\mathbb N:m\ge k,\ N\le\binom{m}{k}\right\}.
\] \end{definition}

Indeed, a set of size \(m\) has exactly \(\binom{m}{k}\) \(k\)-subsets. The minimum is attained by \(\mathcal F_k(N)\).

Let \(n\ge k\) and \(t\ge1\) be integers. We say that \(N\) is in the \emph{\(t\)-th strip}, with respect to \(n\) and \(k\), if \[
\binom{n+t-1}{k}<N\le\binom{n+t}{k}.
\] Equivalently, \(\sigma_k(N)=n+t\). Every family of \(N\) distinct \(k\)-sets in this regime therefore has support-size at least \(n+t\).

\subsection{The first strip}

A problem proposed by Fedor Petrov at the 2016 International Mathematics Competition for University Students \cite{IMC2016} asks for at least three members of size at least \(n\) whenever a union-closed family contains more than \(\binom nk\) \(k\)-sets, with \(n\ge k\). The following theorem gives the sharp bound.

\begin{theorem}\label{thm:first-threshold} Let \(n\ge k\ge1\) be integers. If a union-closed family \(\mathcal F\) contains more than \(\binom nk\) distinct \(k\)-sets, then \[
|\mathcal F_{\ge n}|\ge n-k+3.
\] This bound is attained by a family containing exactly \(\binom nk+1\) \(k\)-sets.
\end{theorem}

At \(N=\binom nk+1\), the max-lex family consists of all the \(k\)-subsets of \([n]\) and one set containing \(n+1\). Its large unions number \(n-k+3\). More generally, the following theorem gives the exact bound throughout the first strip.

\begin{theorem}[First strip]\label{thm:first-strip} Let \(n\ge k\ge1\) and \[
\binom nk<N\le\binom{n+1}{k}.
\] If \(\mathcal F\) is a union-closed family containing exactly \(N\) \(k\)-sets, then \[
|\mathcal F_{\ge n}|
\ge 2+\sigma_{n-k+1}\!\left(N-\binom nk\right)
=U_k(N,n).
\] Equality is attained by \(\langle\mathcal F_k(N)\rangle\).

\end{theorem}

The excess \(R=N-\binom nk\) determines which step of the bound applies: if
\(\binom{s-1}{n-k+1}<R\le\binom{s}{n-k+1}\), the minimum is \(s+2\).
The model count is derived in Section~\ref{the-first-strip}.

The proof begins by taking, for each point \(x\), the union of the generators avoiding \(x\). If this union is too small, many generators contain \(x\), and removing \(x\) reduces both the generator size and the cutoff by one. Otherwise these unions are all large and can be counted. Two points give the same union precisely when they belong to the same generators. The sizes of these incidence classes give one bound using an antichain, a family of sets in which no member properly contains another. A second bound identifies each generator by a small collection of large unions containing it; these identifying collections also form an antichain. Section~\ref{antichains-and-small-certificates} develops the two counts, and Section~\ref{the-first-strip} applies them to complementary parts of the first strip.

\subsection{Minimum support}

In the \(t\)-th strip, the max-lex family has support of size \(n+t\), the least possible size. We first compare it with families on a ground set of that size.

\begin{theorem}[Second strip on minimum support]\label{thm:second-minimum}\label{cor:second-minimum} Let \(n\ge k\ge1\) and \[
\binom{n+1}{k}<N\le\binom{n+2}{k}.
\] If a union-closed family \(\mathcal F\) has support-size \(n+2\) and contains exactly \(N\) \(k\)-sets, then \[
|\mathcal F_{\ge n}|\ge U_k(N,n).
\] Equality is attained by \(\langle\mathcal F_k(N)\rangle\).
\end{theorem}

On an \((n+2)\)-element ground set, a large union omits at most two points. Complementation therefore turns the problem into counting intersections of sizes zero, one, and two. The excess of generators above \(\binom{n+1}{k}\) forces linear relations among the vectors recording the subsets obtained by deleting one point from each complementary set. A parity argument shows that every proper subset of a set participating in one of these relations is itself an intersection. The Björner--Kalai inequality then bounds their number. Points that fail to occur as singleton intersections require a separate deletion argument. Together these counts give exactly the second-strip bound; Section~\ref{small-intersections} proves this in a form that also identifies the remaining deficits in higher strips.

For higher strips, we obtain the following complete range on minimum support.

\begin{theorem}[Higher strips on minimum support]\label{cor:higher-linear}
Let \(k\ge3\), \(t\ge3\), and
\[
 n\ge(t+1)(k-1),\qquad
 \binom{n+t-1}{k}<N\le\binom{n+t}{k}.
\]
If a union-closed family \(\mathcal F\) has support-size \(n+t\) and contains exactly \(N\) \(k\)-sets, then
\[
 |\mathcal F_{\ge n}|\ge U_k(N,n).
\]
Equality is attained by \(\langle\mathcal F_k(N)\rangle\).
\end{theorem}

Here the useful objects are the sets that are not generated. Such a set \(U\) has a point \(z\) uncovered by the generators contained in \(U\). All \(k\)-subsets of \(U\) containing \(z\) must therefore be missing. Counting these missing generators limits the number of missing large unions. Section~\ref{sec:higher-minimum} proves a criterion in terms of the number of missing generators and deduces Theorem~\ref{cor:higher-linear}. It also gives exact bounds for binomial values of that number and proves the complete third strip for triples on minimum support.

\subsection{Larger supports}

On a larger support, the complements of large unions need no longer be small. We instead reduce the support and count the additional unions obtained when deleted points are restored. The next theorem gives a range in which this count suffices for every support size.

\begin{theorem}[General strips]\label{thm:general-strips} Let \(k\ge3\), \(t\ge2\), and let \(n,N\) be positive integers satisfying \[
n>\frac52 k^2t,
\qquad
\binom{n+t-1}{k}<N\le\binom{n+t}{k}.
\] Every union-closed family \(\mathcal F\) containing exactly \(N\) \(k\)-sets satisfies \[
|\mathcal F_{\ge n}|\ge U_k(N,n).
\] Equality is attained by \(\langle\mathcal F_k(N)\rangle\).
\end{theorem}

The proof uses the support-reduction strategy of Ranđelović \cite[Section~2]{Randjelovic2026}, keeping track of the size cutoff. If the support is very large, choose an inclusion-minimal collection of generators covering it. Each chosen generator has a point in no other chosen generator, so different omissions give different unions. On a smaller support, successive deletions of points contained in few generators leave a dense family. We count unions in this remaining family and recover further unions by restoring deleted points. Each restored point permits one more omission from the remaining support without falling below the cutoff. Section~\ref{sec:larger-supports} carries out this comparison and gives the following improvement for large \(k\). Here \(\log\) denotes the natural logarithm.

\begin{theorem}\label{thm:logarithmic-strips}
For every \(\varepsilon>0\) there is an integer \(k_0(\varepsilon)\) such that, for all \(k\ge k_0(\varepsilon)\) and \(t\ge2\), Conjecture~\ref{conj:layered} holds throughout the \(t\)-th strip whenever
\[
 n\ge\frac{((2+\varepsilon)t+1)k^2}{\log k}.
\]
After increasing \(k_0(\varepsilon)\), the same conclusion holds under the simpler condition
\[
 n\ge\frac{(2+\varepsilon)k^2t}{\log k}
 \qquad\text{if also}\qquad t\ge2/\varepsilon.
\]
\end{theorem}

The second strip allows a more precise use of deletion when the generators are small.

\begin{theorem}[Second strip for small generators]\label{thm:small-k}
Let $n\ge k$ and
\[
 \binom{n+1}{k}<N\le\binom{n+2}{k}.
\]
Every union-closed family $\mathcal F$ containing exactly $N$ $k$-sets
satisfies
\[
 |\mathcal F_{\ge n}|\ge U_k(N,n)
\]
in each of the following cases:
\begin{enumerate}
\item[(i)] $k=3$, for every $n\ge3$;
\item[(ii)] $k=4$ and $n\ge40$, or $k=5$ and $n\ge50$.\footnote{A more detailed count gives $n\ge18$ for $k=4$ and $n\ge28$ for $k=5$.}
\end{enumerate}
Equality is attained by $\langle\mathcal F_k(N)\rangle$.
\end{theorem}

Section~\ref{sec:support-induction} considers a different reduction:
identifying two support points. This may reduce the number of distinct
generators and cause some unions to fall below the cutoff. A compensation
conjecture asks for enough additional unions to offset the resulting loss
in the model bound. Theorem~\ref{thm:ind-conditional} shows that, for a
fixed \(k\), this conjecture together with the minimum-support bound in
\emph{every} strip implies the layered conjecture on arbitrary support
in every strip. The section also proves deletion and trace estimates
used in the proof for triples.

The link at a point consists of the generators containing it, with that
point removed. For triples this is a graph, and two of its edges give
an injection that restores the point while raising the cutoff. After
a second deletion, we also use the traces of the surviving triples
outside a chosen generator, meaning their parts outside that set.
The surviving triples and these traces supply disjoint collections of
large unions. Bounds at lower cutoffs control the trace levels; an
elementary count for seven to ten triples starts the induction.
Section~\ref{sec:second-small} gives the resulting proof for every
cutoff. For four-sets and five-sets, the dense-family reduction of
Section~\ref{sec:larger-supports} gives the stated thresholds by a shorter
count.

\hypertarget{antichains-and-small-certificates}{%
\section{Antichains and small certificates}\label{antichains-and-small-certificates}}

The first-threshold bound rests on a simple alternative: either many generators contain a common point, or the unions avoiding individual points already give enough large members. We begin with its proof. The complete first-strip theorem will require two refinements of the second part of this argument.

\begin{proof}[Proof of Theorem~\ref{thm:first-threshold}] We induct on \(k\). For \(k=1\), choose \(n+1\) distinct singletons in \(\mathcal F\). Their union, together with the \(n+1\) unions obtained by omitting one singleton, gives the required \(n+2\) members.

Suppose that \(k\ge2\). Let \(\mathcal A\) consist of the \(k\)-sets in \(\mathcal F\), put \(V=\operatorname{supp}(\mathcal A)\), and, for \(x\in V\), write \[
V_x=\bigcup_{\substack{A\in\mathcal A\\x\notin A}}A.
\]

If \(|V_x|<n\) for some \(x\), at most \(\binom{n-1}{k}\) members of \(\mathcal A\) avoid \(x\). Hence at least \[
\binom nk+1-\binom{n-1}{k}
=
\binom{n-1}{k-1}+1
\] members contain \(x\). Remove \(x\) from these sets and apply induction to their union-closure, with \(n-1\) in place of \(n\). This gives at least \(n-k+3\) unions of size at least \(n-1\). Restoring \(x\) gives the same number of distinct members of \(\mathcal F\) of size at least \(n\).

We may therefore suppose that every \(V_x\) has size at least \(n\). To count the distinct sets \(V_x\), place two points in the same class when they belong to exactly the same members of \(\mathcal A\). Let \(q\) be the number of classes. Points in the same class give the same \(V_x\). Conversely, if a member contains \(x\) but not \(y\), then \(x\in V_y\) and \(x\notin V_x\), so the corresponding unions differ. Thus there are \(q\) distinct sets \(V_x\). Each belongs to \(\mathcal F\) and differs from \(V\), giving \(q+1\) large members when \(V\) is included.

It remains to bound \(q\) from below. Every generator is a union of whole classes. Label the classes \(1,\ldots,q\), and replace each point of a generator by its class label. The resulting multiset has size \(k\) and determines the generator uniquely. Since there are \(\binom{q+k-1}{k}\) such multisets, \[
\binom nk+1\le|\mathcal A|\le\binom{q+k-1}{k}.
\] Consequently \(q+k-1>n\), and hence \(q+1\ge n-k+3\). To see sharpness, take all \(k\)-subsets of \([n]\) and the set \([k-1]\cup\{n+1\}\). The large unions are \([n]\), \([n+1]\), and the \(n-k+1\) sets obtained from \([n+1]\) by omitting a point of \([n]\setminus[k-1]\). Thus the bound is attained.
\end{proof}

\subsection{Counting class selections}

We now refine the count of generators in two ways. First, we retain the sizes of the classes just used. Second, we identify each generator by a small collection of large unions. In each case, the objects being counted form an antichain.

Recall that an \emph{antichain} is a family of sets in which no member properly contains another. For an antichain \(\mathcal H\) of subsets of \([q]\), the LYM inequality states that \[
\sum_{I\in\mathcal H}\binom{q}{|I|}^{-1}\le1.
\] Indeed, the initial segments of a permutation are nested, so at most one belongs to \(\mathcal H\). A fixed set \(I\) is an initial segment in \(|I|!(q-|I|)!\) of the \(q!\) permutations. Counting these occurrences proves the inequality; this is Lubell's argument \cite{Lubell1966}. In particular, \[
|\mathcal H|\le\binom{q}{\lfloor q/2\rfloor}.
\] If every member has size at most \(h\) and \(2h\le q+1\), the same argument gives \(|\mathcal H|\le\binom qh\).

Suppose that the classes have sizes \(a_1,\ldots,a_q\). A generator of size \(j\) is determined by a selection of classes whose sizes add to \(j\). We therefore need a bound for \[
N_j(a_1,\ldots,a_q)
=\bigl|\{I\subseteq[q]:\textstyle\sum_{i\in I}a_i=j\}\bigr|.
\] Here \(j\) is an integer and the weights \(a_i\) are positive integers. Thus \(N_0=1\), and \(N_j=0\) for \(j<0\). The selections counted by \(N_j\) form an antichain, since every weight is positive. The following bound will be used when the number of classes is small.

\begin{lemma}[Weighted antichains]\label{lem:weighted-antichains} Let \(r,q,j\) be integers with \(r\ge0\), \(q>2r\), and \(j>r\). Then \[
N_j(a_1,\ldots,a_q)
\le\binom{q+j-r-2}{j}+\binom{q-1}{r}
\] for all positive integers \(a_1,\ldots,a_q\).

\end{lemma}

\begin{proof} We induct on \(r\). For positive \(r\), a second induction on the length of the weight list will suffice. The common recurrence separates selections according to whether they use the last weight. Put \(a=a_q\). Then
\[
N_j(a_1,\ldots,a_q)
=N_j(a_1,\ldots,a_{q-1})
+N_{j-a}(a_1,\ldots,a_{q-1}).
\]
For the initial case \(r=0\), we use induction on \(q+j\). The cases \(q\le2\) follow from the antichain property, and the case \(j=1\) counts only singleton selections. For \(q\ge3\) and \(j\ge2\), induction bounds the first summand by \(\binom{q+j-3}{j}+1\). If \(a<j\), it bounds the second by
\[
\binom{q+j-a-3}{j-a}+1
\le\binom{q+j-4}{j-1}+1
\le\binom{q+j-3}{j-1}.
\]
Here the first comparison fixes the complementary index \(q-3\), and the second is Pascal's identity. If \(a\ge j\), the second summand is at most one. Adding proves the case \(r=0\).

Now suppose that \(r\ge1\), with the result known for smaller \(r\). We induct on \(q\). The first value is \(q=2r+1\), where the LYM inequality gives \(N_j\le\binom{2r+1}{r}\). The proposed upper bound equals this at \(j=r+1\) and is nondecreasing for larger \(j\).

Let \(q\ge2r+2\). Induction with parameter \(r\) bounds the first summand by
\[
\binom{q+j-r-3}{j}+\binom{q-2}{r}.
\]
We claim that the second summand is at most
\[
\binom{q+j-r-3}{j-1}+\binom{q-2}{r-1}.
\]
For \(j-a\ge r\), this follows from induction with parameter \(r-1\), using the same fixed-complementary-index comparison as above. For \(0\le j-a<r\), every selection has at most \(r-1\) elements, so the LYM inequality gives the bound \(\binom{q-1}{r-1}\). This is no larger than the displayed expression: use \(q-1\ge2r+1\), \(j\ge r+1\), and Pascal's identity. For \(j-a<0\), the summand is zero. Adding the two estimates and applying Pascal's identity proves the lemma.
\end{proof}

\subsection{Identifying generators}

We next count generators by identifying them from large unions. For this purpose, let \(\mathcal A\) be a nonempty family of \(k\)-sets and let \(n=k+p\), where \(p\ge1\). These assumptions remain in force for the rest of the section. Put \[
V=\operatorname{supp}(\mathcal A),
\qquad \mathcal B=\langle\mathcal A\rangle_{\ge n}.
\] As in the proof of Theorem~\ref{thm:first-threshold}, for \(x\in V\) write \[
V_x=\bigcup_{\substack{A\in\mathcal A\\x\notin A}}A.
\] Call \(x\) \emph{bad} if \(|V_x|<n\). When there are no bad points, the large unions distinguish every generator \(A\): for each \(x\in V\setminus A\), the union \(V_x\) contains \(A\) and excludes \(x\). Rather than use all these unions, we shall choose at most \(p+1\) inclusion-minimal ones.

\begin{definition}\label{def:certificate} A \emph{cover} of \(A\in\mathcal A\) is an inclusion-minimal member of \(\mathcal B\) containing \(A\). A collection \(\mathcal S\subseteq\mathcal B\) is a \emph{certificate} for \(A\) if \(A\) is the only member of \(\mathcal A\) contained in every member of \(\mathcal S\).

\end{definition}

\begin{lemma}[Small certificates]\label{lem:small-certificates} Suppose that there are no bad points. Each \(A\in\mathcal A\) has a certificate \(\mathcal S_A\) consisting of at most \(p+1\) covers, all different from \(V\). These certificates form an antichain of subsets of \(\mathcal B\setminus\{V\}\).

\end{lemma}

\begin{proof} Two covers will reduce the possible competing generators to a set of fewer than \(n\) points. Further covers will then exclude the unwanted points individually.

We first show that \(A\) has two distinct covers. At least one exists, since \(V\) is a large union containing \(A\). If \(B_0\) were the unique cover, choose \(x\in B_0\setminus A\); such a point exists because \(n>k\). The large union \(V_x\) contains \(A\) and therefore contains a cover of \(A\). This cover excludes \(x\), contradicting uniqueness. Thus there are at least two covers, and their minimality implies that none is \(V\).

Choose distinct covers \(B_1,B_2\) of \(A\), and put \[
K=\bigcup\{A'\in\mathcal A:A'\subseteq B_1\cap B_2\}.
\] The generators still indistinguishable from \(A\) after choosing \(B_1\) and \(B_2\) all lie in \(K\). This is a generated union containing \(A\). If \(|K|\ge n\), the minimality of the covers would give \(K=B_1=B_2\). Hence \(|K|\le n-1\), so at most \(p-1\) points of \(K\) lie outside \(A\).

We exclude these remaining points one at a time. For each \(x\in K\setminus A\), choose a cover \(B_x\) of \(A\) contained in \(V_x\). Then \[
\mathcal S_A=\{B_1,B_2\}\cup\{B_x:x\in K\setminus A\}
\] has at most \(p+1\) members. Any generator contained in all of them lies in \(K\) and avoids \(K\setminus A\). It is therefore contained in \(A\), and equals \(A\) because all generators have size \(k\). Thus \(\mathcal S_A\) is a certificate. Finally, if \(\mathcal S_A\subseteq\mathcal S_{A'}\), then \(A'\) is contained in every member of \(\mathcal S_A\), so \(A'=A\). This proves the lemma.
\end{proof}

The certificates avoid the full support. To sharpen their count, we shall also set aside one proper large union. The generators for which that union is a cover must be counted separately.

\begin{lemma}[The certificate bound]\label{lem:certificate-bound} Suppose that there are no bad points, and let \(q\ge2p+1\) be an integer. If \(|\mathcal B|\le q+2\), then \[
|\mathcal A|\le\binom nk+\binom{q}{p+1}.
\]

\end{lemma}

\begin{proof} Choose \(B_*\in\mathcal B\setminus\{V\}\); such a member exists by the preceding lemma. Let \(\mathcal D\) consist of the generators for which \(B_*\) is a cover. Any large union of members of \(\mathcal D\) is contained in \(B_*\) and contains one of these generators. By minimality of the cover, it must equal \(B_*\). Thus \(\langle\mathcal D\rangle\) has at most one large member, and Theorem~\ref{thm:first-threshold} gives \(|\mathcal D|\le\binom nk\).

For every remaining generator, Lemma~\ref{lem:small-certificates} supplies a certificate using neither \(V\) nor \(B_*\), since every member of the certificate is a cover. These certificates form an antichain of sets of size at most \(p+1\), drawn from at most \(q\) large unions. The assumption \(q\ge2p+1\) allows the bounded-size form of the LYM inequality, giving \(|\mathcal A\setminus\mathcal D|\le\binom q{p+1}\). Adding the bound for \(\mathcal D\) proves the lemma.

\end{proof}

\hypertarget{the-first-strip}{%
\section{The first strip}\label{the-first-strip}}

We now prove Theorem~\ref{thm:first-strip}. As in the first-threshold proof, we remove bad points until every union avoiding a point is large. This reduction preserves both the difference between the cutoff and the generator size, and the lower bound on the excess number of generators. The two estimates of Section~\ref{antichains-and-small-certificates} then give the successive steps of the max-lex count.

\hypertarget{the-max-lex-count}{%
\subsection{The max-lex count}\label{the-max-lex-count}}

We first compute the number predicted by the conjecture. Put \(p=n-k\) and assume that \(N>\binom nk\). In the \(t\)-th strip, write \[
N=\binom{n+t-1}{k}+R,
\qquad 1\le R\le\binom{n+t-1}{k-1}.
\] Thus \(R\) counts the generators containing the largest point \(n+t\). As these generators are added, the number of large unions increases in steps. We describe this staircase, beginning with the first strip.

In the first strip, \(\mathcal F_k(N)\) contains all the \(k\)-subsets of \([n]\) and \(R\) sets containing \(n+1\). Deleting \(n+1\) from the latter gives the first \(R\) \((k-1)\)-subsets of \([n]\) in lexicographic order. For \(p+1\le s\le n\), the \(\binom{s}{p+1}\) \((k-1)\)-subsets containing \([n-s]\) come first in this order. Hence the first \(R\) sets have common intersection \([n-s]\) when \[
\binom{s-1}{p+1}<R\le\binom{s}{p+1}.
\]

An \(n\)-subset of \([n+1]\) containing \(n+1\) is generated exactly when it contains one of the generators through \(n+1\). Equivalently, its missing point must lie outside the common intersection just found. There are \(s\) such sets. The two other large members are \([n]\) and \([n+1]\). Therefore, on this interval, \[
U_k(N,n)=s+2=2+\sigma_{p+1}(R).
\] The first strip has \(k\) steps.

Similarly, in the second strip, \(R=N-\binom{n+1}{k}\). For \(p+1\le r<s\le n+1\), we have \[
U_k(N,n)=n+4+\binom{s}{2}+r
\] on the interval \[
\binom{s-1}{p+2}+\binom{r-1}{p+1}
<R\le
\binom{s-1}{p+2}+\binom{r}{p+1}.
\] There are \(\binom{k+1}{2}\) steps.

For the general formula, we use the binomial sums \[
S_j(a)=\sum_{\ell=0}^{j}\binom{a}{\ell}
\qquad(a,j\text{ nonnegative integers}).
\] Roberts's incremental counts for max-lex unions are recorded in
McLeod's thesis \cite[Theorem~5.15]{McLeod2002}. The following formula
groups the counts into steps in a general strip.

\Needspace{12\baselineskip}
\begin{lemma}\label{lem:model-count} Let \(n\ge k\ge1\), \(N>\binom nk\), \(p=n-k\), \(t=\sigma_k(N)-n\), and \(R=N-\binom{n+t-1}{k}\). For integers \(p+1\le s_1<s_2<\cdots<s_t\le n+t-1\), put \[
B=\sum_{i=1}^{t}\binom{s_i-1}{p+i}.
\] On the interval \(B<R\le B+\binom{s_1-1}{p}\), we have \[
U_k(N,n)
=1+S_{t-1}(n+t-1)+\sum_{i=1}^{t}S_i(s_i-1).
\] These intervals partition the \(t\)-th strip into \(\binom{k+t-1}{t}\) steps.

\end{lemma}

\begin{proof}
Put $m=n+t$ and $q=m-k=p+t$. The generators avoiding $m$ form
the complete $k$-layer on $[m-1]$, and give $S_{t-1}(m-1)$ large
unions. A set containing $m$ is generated if and only if it contains
one of the generators through $m$. Indeed, after choosing that
generator, the other points can be covered by $k$-sets avoiding $m$;
at size $k$, the set must itself be a generator.

Delete $m$ from the remaining generators, complement in $[m-1]$,
and reverse the labels. Their images are the first $R$ $q$-sets in
\emph{colex order}: of two sets, the one excluding the largest point
where they differ comes first. Unions through $m$ of size at least
$n$ now correspond to subsets of size at most $t$ of these $q$-sets.

The colex $q$-sets with fixed $t$ largest elements
$s_1<\cdots<s_t$ form a block. Its first member is
$[p]\cup\{s_1,\ldots,s_t\}$, its preceding members number $B$,
and its length is $\binom{s_1-1}{p}$. These statements follow by
successively fixing the largest element in the colex order.

The $j$-subsets of an initial colex segment themselves form an
initial colex segment. For $1\le j\le t$, its last member is the set
of the $j$ largest elements of the last $q$-set. Hence, throughout
the present block, their number is
\[
 1+\sum_{i=t-j+1}^{t}\binom{s_i-1}{i-t+j}.
\]
Adding these counts for $j=1,\ldots,t$, and then the empty subset,
gives $1+\sum_{i=1}^{t}S_i(s_i-1)$. Together with the unions
avoiding $m$, this is the asserted formula.

Finally, the blocks partition all $q$-sets of $[m-1]$. There are
$\binom{k+t-1}{t}$ choices of their $t$ largest elements. Each block
introduces the new $t$-set $\{s_1,\ldots,s_t\}$, so each starts a
new step. This proves both the formula and the partition.
\end{proof}

\subsection{Proof of the first-strip bound}

\begin{proof}[Proof of Theorem~\ref{thm:first-strip}] If \(n=k\), the \(N\) generators and their full union give the required \(N+1\) members. The case \(N=\binom nk+1\) is Theorem~\ref{thm:first-threshold}. We may therefore assume that \(n>k\) and \(R=N-\binom nk\ge2\). Let \(\mathcal A\) consist of the \(k\)-sets in \(\mathcal F\), and put \[
p=n-k,\qquad s=\sigma_{p+1}(R),\qquad
h=\binom{s-1}{p+1}.
\] Thus \(s\) specifies the step containing \(R\), and \(h\) is the last excess before this step begins. We have \(p+2\le s\le n\) and \(R\ge h+1\). We must find at least \(s+2\) large unions.

\emph{Reduction to a family with no bad points.} During the reduction, write \(\kappa\) for the size of the generators and \(\lambda=\kappa+p\) for the cutoff. We keep at least \(\binom\lambda\kappa+h+1\) generators. If \(x\) is bad, at most \(\binom{\lambda-1}{\kappa}\) of them avoid \(x\). The \emph{link at \(x\)}, obtained by deleting \(x\) from the generators containing it, therefore has at least \[
\binom\lambda\kappa+h+1-\binom{\lambda-1}{\kappa}
=\binom{\lambda-1}{\kappa-1}+h+1
\] members. The generator size and the cutoff both decrease by one, while \(p\) and \(h\) remain fixed. Restoring \(x\) sends distinct unions of size at least \(\lambda-1\) to distinct unions of size at least \(\lambda\). Thus a lower bound for the link gives the same lower bound for the preceding family.

The process stops before the generator size reaches zero. Indeed, when \(\kappa=1\), a bad point would have to belong to at least \(h+2\) distinct singletons, which is impossible. We therefore reach a family \(\mathcal C\) of \(\kappa\)-sets, with cutoff \(\lambda=\kappa+p\), no bad points, and \begin{equation}\label{eq:reduced-first-strip}
|\mathcal C|\ge\binom\lambda\kappa+h+1.
\end{equation} It remains to find \(s+2\) unions of size at least \(\lambda\) in this family. Write \(V=\operatorname{supp}(\mathcal C)\) and \(\mathcal B=\langle\mathcal C\rangle_{\ge\lambda}\). Suppose, for a contradiction, that \(|\mathcal B|\le s+1\).

The certificate bound requires \(s-1\ge2p+1\). We use class selections for the earlier steps and certificates for the later ones.

\emph{Case \(s\le2p+1\).} Put \(r=s-p-2\), so that \(0\le r<p\) and \[
h=\binom{s-1}{r}.
\] Group the points of \(V\) according to which generators contain them, as in Theorem~\ref{thm:first-threshold}, and let \(q\) be the number of classes. Every generator is a union of whole classes. For a class \(D\), write \(V_D=V_x\) for any \(x\in D\). The \(q\) sets \(V_D\) are distinct proper members of \(\mathcal B\). Including \(V\) gives \(q+1\) large unions, so \(q\le s\).

If \(\kappa>r\), extend the list of class sizes to length \(s\) by adding weights \(\kappa+1\). These weights cannot enter a selection of total size \(\kappa\). Lemma~\ref{lem:weighted-antichains} applies to the resulting list, since \(s=p+r+2>2r\), and gives \[
|\mathcal C|
\le\binom{s+\kappa-r-2}{\kappa}+\binom{s-1}{r}
=\binom\lambda\kappa+h,
\] contrary to \eqref{eq:reduced-first-strip}.

The remaining case \(\kappa\le r\) lies outside Lemma~\ref{lem:weighted-antichains}. Here we first show that the already counted unions exhaust \(\mathcal B\). For if \(q\le s-1\), the class selections forming generators are an antichain on at most \(s-1\) objects. Each has at most \(\kappa\) members, and \(2\kappa\le2r<s-1\). The LYM inequality gives \[
|\mathcal C|\le\binom{s-1}{\kappa}
\le\binom{s-1}{r}=h,
\] again a contradiction. Thus \(q=s\), and the \(s+1\) sets \(V,V_D\) account for every member of \(\mathcal B\).

We now find a large union outside this list. Fix a class \(D\) and discard all generators containing it. There are at most \(h\) such generators: after removing \(D\), their remaining class selections form an antichain on \(s-1\) objects, each with at most \(\kappa-1\le r\) members. The LYM inequality bounds their number by \(\binom{s-1}{r}=h\). By \eqref{eq:reduced-first-strip}, more than \(\binom\lambda\kappa\) generators survive. Theorem~\ref{thm:first-threshold} supplies at least \[
\lambda-\kappa+3=p+3
\] large unions disjoint from \(D\).

At most \(\kappa\) unions in the list \(V,V_E\) are disjoint from \(D\). Indeed, \(V\) meets \(D\). Fix a generator containing \(D\). If \(V_E\) is disjoint from \(D\), this generator must contain \(E\); otherwise it would be included in \(V_E\). A generator contains at most \(\kappa\) classes, so there are at most \(\kappa\) possible classes \(E\). Since \(\kappa\le r<p\), the \(p+3\) unions just found cannot all occur in the list. This is a contradiction.

\emph{Case \(s\ge2p+2\).} The certificate count now applies. Use Lemma~\ref{lem:certificate-bound} for \(\mathcal C\) with cutoff \(\lambda\) and \(q=s-1\). Its hypotheses hold because \(q\ge2p+1\) and \(|\mathcal B|\le s+1=q+2\). It gives \[
|\mathcal C|\le\binom\lambda\kappa+\binom{s-1}{p+1}
=\binom\lambda\kappa+h,
\] which contradicts \eqref{eq:reduced-first-strip}.

Thus \(|\mathcal B|\ge s+2\) in every case. Restoring the deleted points gives at least \(s+2\) members of \(\mathcal F_{\ge n}\). The max-lex count gives equality, and Theorem~\ref{thm:first-strip} is proved.

\end{proof}

\hypertarget{small-intersections}{%
\section{Small intersections}\label{small-intersections}}

Taking complements on \(n+t\) points turns unions of size at least
\(n\) into intersections of size at most \(t\). We shall count these
intersections and prove Theorem~\ref{thm:second-minimum} by taking
\(t=2\).

There are two steps. First, the excess number of complementary sets
forces linear relations among the vectors recording their one-point
deletions. Each set occurring in
such a relation has all its proper subsets realised as intersections.
A shadow inequality of Björner and Kalai then counts these subsets.
Second, we must account for ground-set points which are not themselves
singleton intersections. Deleting such a point either produces a
realised pair in its place or increases the excess enough to force an
additional intersection.

\begin{definition}\label{def:small-intersections} For a family \(\mathcal J\), let \(\mathcal K(\mathcal J)\) consist of the intersections of nonempty subfamilies of \(\mathcal J\). We call these the \emph{realised intersections}. In particular, the empty set is included only when some nonempty subfamily has empty intersection. We use the level and tail subscripts defined in the introduction.

The numerical bounds will use counts of subsets of an initial colex segment, as in the proof of Lemma~\ref{lem:model-count}. Recall that \(A\) precedes \(B\) in \emph{colex order} if the largest integer belonging to exactly one of them belongs to \(B\). Let \(\mathcal I_q(R)\) be the first \(R\) \(q\)-sets in this order, where \(q\ge1\) and \(R\ge0\).

The \emph{\(j\)-shadow} of a family consists of the \(j\)-sets contained in its members. For \(0\le j\le q\), write \[
h_j^{(q)}(R)
=\bigl|\{B:|B|=j,\ B\subseteq A
\text{ for some }A\in\mathcal I_q(R)\}\bigr|.
\] All these functions vanish at \(R=0\), and \(h_0^{(q)}(R)=1\) for \(R>0\).
\end{definition}

The functions \(h_j^{(q)}\) are nondecreasing. The shadow of an initial colex segment is again an initial colex segment. Taking shadows in two stages therefore gives \begin{equation}\label{eq:shadow-composition}
h_i^{(j)}\!\left(h_j^{(q)}(R)\right)=h_i^{(q)}(R)
\qquad(1\le i\le j\le q).
\end{equation}

\Needspace{13\baselineskip}\begin{lemma}[Intersections of a fixed size]\label{lem:cycle-intersections} Let \(m>q\ge3\) be integers and let \(\mathcal J\subseteq\binom{[m]}q\), with \[
|\mathcal J|=\binom{m-1}{q-1}+R,
\qquad R\ge1.
\] For every \(2\le j<q\), \[
|\mathcal K(\mathcal J)_j|
\ge h_j^{(q)}(R)+h_{j-1}^{(q)}(R).
\]

\end{lemma}

\begin{proof} We first obtain \(R\) independent relations, then show
that the sets occurring in these relations supply realised intersections,
and finally count them by their shadows.

\emph{The relations.} Work over the field \(\mathbb F_2\) with two elements. Regard all
subsets of \([m]\) as basis vectors, and define a linear map by \[
\partial A=\sum_{x\in A}(A\setminus\{x\}),
\qquad(A\subseteq[m]).
\] This is the \emph{boundary} of \(A\); in particular,
\(\partial\varnothing=0\).
Each set obtained by deleting two points occurs
twice, so \(\partial^2=0\). A vector in the kernel of \(\partial\) is
called a \emph{cycle}. In a cycle, every set of one smaller size occurs
an even number of times as a subset of the participating sets.

Fix \(v\in[m]\). For a \(q\)-set \(A\) avoiding \(v\), cancellation
of the terms containing \(v\) gives
\[
 \partial A=\sum_{x\in A}
 \partial\bigl((A\setminus\{x\})\cup\{v\}\bigr).
\]
Thus the boundaries of the \(q\)-sets through \(v\) span all the
boundaries. Their rank is at most \(\binom{m-1}{q-1}\), and the space
of cycles supported on \(\mathcal J\) has dimension at least \(R\).

\emph{Realising the subsets.} Let \(A\) occur in a cycle and let \(F\subsetneq A\). For each
\(x\in A\setminus F\), the parity condition supplies another
participating member \(A_x\) containing \(A\setminus\{x\}\).
Since \(A_x\ne A\), it excludes \(x\). Consequently \[
F=A\cap\bigcap_{x\in A\setminus F}A_x.
\] Every proper subset of \(A\) is therefore a realised intersection.

\emph{Counting the subsets.} Let \(\Delta\) consist of all subsets of members of \(\mathcal J\) that occur in at least one cycle. This is a \emph{simplicial complex}: it is closed under taking subsets. Its members are called \emph{faces}. We have shown that every face of size less than \(q\) is a realised intersection.

Let \(C_j(\Delta)\) be the vector space over \(\mathbb F_2\) with the \(j\)-element faces as a basis, including the empty face when \(j=0\). For the boundary map \[
\partial_j:C_j(\Delta)\longrightarrow C_{j-1}(\Delta)
\qquad(j\ge1),
\] write \(\rho_j\) for its rank and \(z_j\) for its nullity. The number of \(j\)-element faces is therefore \(\rho_j+z_j\). Since every boundary is a cycle, \(z_{j-1}\ge\rho_j\) for \(j\ge2\). All the original cycles are supported on \(\Delta\), so \(z_q\ge R\).

The Björner--Kalai inequality compares these ranks and nullities with
colex shadows: \begin{equation}\label{eq:boundary-shadow}
\rho_j\ge h_{j-1}^{(j)}(z_j)
\qquad(j\ge2).
\end{equation} It bounds the number of independent boundaries in terms of the number of independent cycles \cite[Theorem~1.1]{BjornerKalai1988}; see also \cite[Theorem~1.1 and Remark~1.6]{ZhanHuang2025}. Those statements index faces by dimension, whereas we index them by cardinality: their cycle dimension in degree $j-1$ is $z_j$, and the boundary map out of that degree has rank $\rho_j$. We use reduced chains, so the boundary of each singleton is the empty face.

Starting with \(z_q\ge R\), inequality \eqref{eq:boundary-shadow} gives \(\rho_q\ge h_{q-1}^{(q)}(R)\). This also bounds \(z_{q-1}\), since \(z_{q-1}\ge\rho_q\). Repeating at each smaller size, and using the composition rule \eqref{eq:shadow-composition}, gives \[
z_j\ge h_j^{(q)}(R),
\qquad
\rho_j\ge h_{j-1}^{(q)}(R)
\qquad(2\le j<q).
\] Thus \(\Delta\) has at least \(h_j^{(q)}(R)+h_{j-1}^{(q)}(R)\) faces of size \(j\). Since all these faces are realised intersections, the lemma follows.
\end{proof}

For example, a single independent cycle forces at least
\(\binom q2+q=\binom{q+1}{2}\) realised pairs. More generally, the lemma
gives the part of the count that depends on \(R\). We must also account
for the empty set and the \(m\) ground-set points. If every singleton
is realised, this is immediate. The next proof handles missing
singletons by deletion. Unused ground-set points are allowed so that
the induction preserves its hypotheses.

\begin{lemma}[Small realised intersections]\label{lem:small-intersections} Let \(m>q\ge3\) and \(2\le t<q\) be integers. If \(\mathcal J\subseteq\binom{[m]}q\) and \[
|\mathcal J|=\binom{m-1}{q-1}+R,
\qquad R\ge1,
\] then \begin{equation}\label{eq:small-intersections}
|\mathcal K(\mathcal J)_{\le t}|
\ge m+1+\sum_{j=2}^{t}
\bigl(h_j^{(q)}(R)+h_{j-1}^{(q)}(R)\bigr).
\end{equation} The ground set \([m]\) need not be the support of \(\mathcal J\).

\end{lemma}

\begin{proof} For this proof, write \(\Psi(R)\) for the sum on the right of \eqref{eq:small-intersections}. It is nondecreasing in \(R\). We induct on \(m\), keeping \(q\) and \(t\) fixed.

When \(m=q+1\), the hypothesis forces \(\mathcal J\) to be the complete family of \(q\)-sets and \(R=1\). Every subset of size at most \(t\) is realised. Since \(h_j^{(q)}(1)=\binom qj\), Pascal's identity gives equality in \eqref{eq:small-intersections}.

Suppose next that every point is a realised singleton. Intersecting two of these also realises the empty set. These \(m+1\) intersections and the \(\Psi(R)\) supplied by Lemma~\ref{lem:cycle-intersections} give the required bound.

Otherwise choose a point \(x\) that is not a realised singleton. Let \(d_x\) be the number of members containing \(x\), and let \(\mathcal J'\) consist of the members avoiding \(x\), on the remaining \(m-1\) points. The excess for this smaller ground set is \begin{equation}\label{eq:intersection-excess}
R'=R+\binom{m-2}{q-2}-d_x,
\qquad
|\mathcal J'|=\binom{m-2}{q-1}+R'.
\end{equation}

Induction on the remaining points will give \(m+\Psi(R')\), whereas we need \(m+1+\Psi(R)\). We therefore need either one intersection through \(x\), or an increase in \(\Psi\). If \(x\) occurs in a member, let \(C_x\) be the intersection of all members containing it. Since \(\{x\}\) is not realised, \(|C_x|\ge2\).

Suppose first that \(|C_x|=2\). Every member through \(x\) contains this pair, so \(d_x\le\binom{m-2}{q-2}\) and \(R'\ge R\). Induction gives at least \(m+\Psi(R')\) intersections of size at most \(t\) from \(\mathcal J'\). The pair \(C_x\) is one more realised intersection: it contains \(x\), whereas all those just counted avoid \(x\). Since \(\Psi(R')\ge\Psi(R)\), this proves the bound.

It remains to consider \(|C_x|\ge3\) or an unused point \(x\). In the former case, every member through \(x\) contains a fixed three-set; in the latter, \(d_x=0\). Thus in both cases \(d_x\le\binom{m-3}{q-3}\), and \eqref{eq:intersection-excess} gives \[
R'-R\ge\binom{m-3}{q-2}.
\] We shall show that this increase in excess gives \(\Psi(R')\ge\Psi(R)+1\). First note that \[
0<R\le R'\le\binom{m-2}{q},
\] where the last inequality follows from \(|\mathcal J'|\le\binom{m-1}{q}\). Hence the first \(R'\) colex \(q\)-sets lie in \([m-2]\).

In colex order, the \(q\)-sets of \([m-2]\) occur in blocks with fixed two largest elements \(a>b\). Each block has \(\binom{b-1}{q-2}\le\binom{m-4}{q-2}\) members. Its first member introduces the previously absent pair \(\{a,b\}\). Thus \(h_1^{(q)}+h_2^{(q)}\) cannot be constant on more than \(\binom{m-4}{q-2}\) consecutive positive arguments. But \(m-2\ge q\) and \[
R'-R\ge\binom{m-3}{q-2}>\binom{m-4}{q-2}.
\] The sum \(h_1^{(q)}+h_2^{(q)}\) therefore increases between \(R\) and \(R'\). It is a summand of \(\Psi\), and all remaining terms are nondecreasing. This proves \(\Psi(R')\ge\Psi(R)+1\). Induction now gives \[
|\mathcal K(\mathcal J)_{\le t}|
\ge |\mathcal K(\mathcal J')_{\le t}|
\ge m+\Psi(R')
\ge m+1+\Psi(R).
\] This proves the lemma.
\end{proof}

We now apply the lemma to the complementary generators. At $t=2$
the two terms involving shadows agree with the variable part of the
max-lex count.

\begin{proof}[Proof of Theorem~\ref{thm:second-minimum}] Suppose first that \(n>k\). Let \(V\) be the support of \(\mathcal F\), let \(\mathcal A\) consist of its \(k\)-sets, and put \[
m=n+2,\qquad q=m-k,\qquad
R=N-\binom{n+1}{k}.
\] The complementary family \(\mathcal J=\{V\setminus A:A\in\mathcal A\}\) consists of \(q\)-sets and satisfies \[
|\mathcal J|=\binom{m-1}{q-1}+R.
\] Since \(q\ge3\), Lemma~\ref{lem:small-intersections} applies with \(t=2\). Complementation is a bijection between \(\langle\mathcal A\rangle_{\ge n}\) and \(\mathcal K(\mathcal J)_{\le2}\), and therefore \[
|\mathcal F_{\ge n}|
\ge |\langle\mathcal A\rangle_{\ge n}|
\ge n+3+h_1^{(q)}(R)+h_2^{(q)}(R).
\] To identify the last expression, recall that the \(R\) generators through the largest point in the max-lex model are obtained from the first \(R\) lexicographic \((k-1)\)-sets of \([n+1]\). Complementation in \([n+1]\) and reversal of its labels send this segment to \(\mathcal I_q(R)\). Thus the last expression is precisely \(U_k(N,n)\), as counted in subsection~\ref{the-max-lex-count}.

If \(n=k\), the same range of \(N\) is the first strip for cutoff \(k+1\). Apply Theorem~\ref{thm:first-strip} at that cutoff and add the \(N\) members of size \(k\). This gives \[
|\mathcal F_{\ge k}|
=N+|\mathcal F_{\ge k+1}|
\ge N+U_k(N,k+1)
=U_k(N,k).
\] The max-lex family attains the count in both cases. This proves the theorem.

\end{proof}

\begin{remark} The same comparison shows what remains to be counted in higher strips. Let \(n>k\), let \(N\) lie in the \(t\)-th strip, and put \(m=n+t\), \(q=m-k\), and \(R=N-\binom{m-1}{k}\). For a family \(\mathcal A\) of \(N\) \(k\)-sets on \(m\) points, Lemma~\ref{lem:small-intersections} and the model count give \[
|\langle\mathcal A\rangle_{\ge n}|
\ge U_k(N,n)
-\sum_{j=2}^{t-1}
\left[\binom{m-1}{j}-h_j^{(q)}(R)\right]
\qquad(t\ge2).
\] Indeed, the model count is \(S_{t-1}(m-1)+\sum_{j=0}^{t}h_j^{(q)}(R)\). For \(t=2\), the sum subtracted in the displayed bound is empty, giving Theorem~\ref{thm:second-minimum}. For \(t\ge3\), one must also account for these nonnegative deficits at sizes \(2,\ldots,t-1\).

\end{remark}

\hypertarget{higher-strips-on-minimum-support}{%
\section{Higher strips on minimum support}\label{higher-strips-on-minimum-support}\label{sec:higher-minimum}}

On minimum support, a set which is not generated forces many generators
to be missing. We shall count these missing generators and prove
Theorem~\ref{cor:higher-linear}. The same point of view gives exact
values when the number of missing generators is a single binomial
coefficient. In the third strip, keeping track of the generators
through a deleted point will give the complete bound for triples.

Throughout this section, unless another range is stated, let \(n>k\ge2\), \(t\ge1\), and let \(\mathcal A\) be a family of \(N\) \(k\)-sets with support \(V\), where \[
 |V|=m=n+t,\qquad
 \binom{m-1}{k}<N\le\binom mk.
\] Write \[
 h=\binom mk-N,\qquad R=N-\binom{m-1}{k},\qquad q=m-k.
\] Thus \(h+R=\binom{m-1}{k-1}\). We use the shadow functions of Definition~\ref{def:small-intersections}. The model count from subsection~\ref{the-max-lex-count} takes the form \begin{equation}\label{eq:higher-model}
 U_k(N,n)=S_{t-1}(m-1)+\sum_{j=0}^{t}h_j^{(q)}(R).
\end{equation} Indeed, the complete \(k\)-layer on \(m-1\) points supplies the first sum. Complementing the remaining generators turns the count of unions through the last point into the second sum.

\subsection{Blocking points}

A point \(z\in U\subseteq V\) \emph{blocks} \(U\) if no member of
\(\mathcal A\) contained in \(U\) contains \(z\). Thus a nonempty
set fails to be generated precisely when it has a blocking point:
in the absence of such a point, the generators contained in the set
cover it. If \(z\) blocks a \(u\)-set \(U\), all
\(\binom{u-1}{k-1}\) \(k\)-subsets of \(U\) through \(z\) are missing.

The collections of missing generators belonging to different blocking
points may overlap. When two fixed points suffice, the following
shadow argument controls this overlap. For a family \(\mathcal D\),
write \(\partial_j\mathcal D\) for its \(j\)-shadow. We use the
Kruskal--Katona theorem in the form
\(|\partial_j\mathcal D|\ge h_j^{(s)}(|\mathcal D|)\) for a family
of \(s\)-sets \cite{Katona1968,Kruskal1963}.

\begin{lemma}[Two designated points]\label{lem:two-designated}
Let $U_1,\ldots,U_L$ be distinct $u$-sets, and choose $z_i\in U_i$.
Suppose that at most two different points occur among the $z_i$.
For $u\ge k\ge2$, their \emph{pointed $k$-shadow} is
\[
 \mathcal P_k=\{A:|A|=k,\ z_i\in A\subseteq U_i
                  \text{ for some }i\}.
\]
Then $|\mathcal P_k|\ge h_{k-1}^{(u-1)}(L)$.
\end{lemma}

\begin{proof}
The assertion is immediate if $L=0$.
If all the designated points coincide, delete that point and apply
Kruskal--Katona. Otherwise call the two points $a,b$, and put
\[
 \mathcal W_a=\{U_i\setminus\{a\}:z_i=a\},\qquad
 \mathcal W_b=\{U_i\setminus\{b\}:z_i=b\}.
\]
Their total size is $L$. The common members
$\mathcal D=\mathcal W_a\cap\mathcal W_b$ avoid both $a$ and $b$;
write $d=|\mathcal D|$.

We shall move the remaining members of $\mathcal W_b$ to the ground set
avoiding $a$. For $W\in\mathcal W_b\setminus\mathcal D$, replace $W$
by $W$ if $a\notin W$, and by $(W\setminus\{a\})\cup\{b\}$ otherwise.
Call the resulting family $\mathcal W_b^*$.
These replacements are distinct. None belongs to $\mathcal W_a$:
a coincidence in the first case would lie in $\mathcal D$, and a
coincidence in the second would give the same $U_i$ two designations.
Thus $\mathcal W_a\cup\mathcal W_b^*$ has $L-d$ members.

Here is an injection from its $(k-1)$-shadow into $\mathcal P_k$.
For $E$ in that shadow, send $E$ to $E\cup\{a\}$ if
$E\in\partial_{k-1}\mathcal W_a$ or $b\in E$; otherwise send it to
$E\cup\{b\}$. Each image belongs to $\mathcal P_k$ by the construction
of $\mathcal W_b^*$. Images of the first kind contain $a$, while those
of the second kind do not. Within each kind the map is injective.

There are also $|\partial_{k-1}\mathcal D|$ further members of
$\mathcal P_k$, namely $E\cup\{b\}$ with
$E\in\partial_{k-1}\mathcal D$. They contain no $a$, and are excluded
from the second kind of image because
$\partial_{k-1}\mathcal D\subseteq\partial_{k-1}\mathcal W_a$.
Consequently
\[
 |\mathcal P_k|
 \ge h_{k-1}^{(u-1)}(L-d)+h_{k-1}^{(u-1)}(d)
 \ge h_{k-1}^{(u-1)}(L).
\]
For the last inequality, place extremal families of the two indicated
sizes on disjoint ground sets and apply Kruskal--Katona to their union.
This proves the lemma.
\end{proof}

\begin{proposition}\label{prop:two-blockers}
If every subset of $V$ of size at least $n$ which is not generated has
a blocking point in a fixed set of at most two points, then
\[
 |\langle\mathcal A\rangle_{\ge n}|\ge U_k(N,n).
\]
\end{proposition}

\begin{proof}
Let $b_u$ be the number of $u$-subsets of $V$ which are not generated.
Choose the blocking points as in the hypothesis. Their pointed
$k$-shadows lie among the $h$ missing generators, so
Lemma~\ref{lem:two-designated} gives
$h_{k-1}^{(u-1)}(b_u)\le h$.
The complementary form of Kruskal--Katona rewrites this as
\[
 b_{m-j}\le\binom{m-1}{j}-h_j^{(q)}(R)
 \qquad(1\le j\le t).
\]
To see this form, let $\mathcal H$ be the first $h$ $(k-1)$-sets in
colex order on $m-1$ points. The shadow of the first $b_{m-j}$
$(m-j-1)$-sets is an initial segment of size at most $h$, hence lies
in $\mathcal H$. Thus these $b_{m-j}$ sets each have all their
$(k-1)$-subsets in $\mathcal H$. Complementation reverses colex
order: the complements of the $R$ sets outside $\mathcal H$ form
$\mathcal I_q(R)$. A set has all its $(k-1)$-subsets in $\mathcal H$
exactly when its complementary $j$-set lies outside the $j$-shadow
of $\mathcal I_q(R)$. This gives the displayed bound, also when $h=0$.

The support itself is generated. Subtracting these bounds from the
numbers of subsets of each size gives
\[
 |\langle\mathcal A\rangle_{\ge n}|
 \ge S_t(m)-S_t(m-1)+\sum_{j=0}^{t}h_j^{(q)}(R)
 =U_k(N,n),
\]
by Pascal's identity and \eqref{eq:higher-model}. This proves the
proposition.
\end{proof}

\subsection{The higher-strip bound}

If every subset above size \(n\) is generated, the intersection lemma
already supplies enough \(n\)-sets. Otherwise a missing larger set
forces a larger collection of missing generators. We shall use that
collection to show that two blocking points suffice.

\begin{lemma}\label{lem:bottom-level}
Suppose $t\ge2$. If every subset of $V$ of size at least $n+1$ is
generated, then $|\langle\mathcal A\rangle_{\ge n}|\ge U_k(N,n)$.
\end{lemma}

\begin{proof}
Complement the generators in $V$. Since $q>t$,
Lemma~\ref{lem:cycle-intersections} supplies at least
$h_t^{(q)}(R)+h_{t-1}^{(q)}(R)$ realised $t$-sets, hence that many
generated $n$-sets. Adding the larger subsets gives
\[
 \begin{aligned}
 |\langle\mathcal A\rangle_{\ge n}|
 &\ge S_{t-1}(m)+h_t^{(q)}(R)+h_{t-1}^{(q)}(R)\\
 &=U_k(N,n)+S_{t-2}(m-1)-\sum_{j=0}^{t-2}h_j^{(q)}(R).
 \end{aligned}
\]
The last difference is nonnegative, since $\mathcal I_q(R)$ lies on
$m-1$ points. This proves the lemma.
\end{proof}

\begin{theorem}[The missing-generator criterion]\label{thm:h-criterion}
Suppose $n>k\ge3$ and $t\ge3$. If
\begin{equation}\label{eq:h-criterion}
 h<\binom n{k-1}+\binom{n-2}{k-1}+\binom{n-3}{k-1},
\end{equation}
then $|\langle\mathcal A\rangle_{\ge n}|\ge U_k(N,n)$.
\end{theorem}

\begin{proof}
We shall find at most two points which between them block every set
of size at least $n$ that is not generated. If a third point is needed,
the three blocking collections will contradict the bound on $h$.

By Lemma~\ref{lem:bottom-level}, we may suppose that a set of size at
least $n+1$ is not generated. Choose a blocking point $z$ and retain
an $(n+1)$-subset $U$ containing it. All the $\binom n{k-1}$
$k$-subsets of $U$ through $z$ are missing.

If $z$ blocks every set under consideration, apply
Proposition~\ref{prop:two-blockers}. Otherwise choose a set $W$ of
size at least $n$ which is not generated and is not blocked by $z$,
and choose a blocking point $y$ of $W$. If $y,z$ suffice for all such
sets, the same proposition applies. Suppose they do not, and choose
a set $T$ blocked by neither, with blocking point $x$.
The three points $x,y,z$ are distinct.

Choose distinct $n$-subsets $W_0\subseteq W$, $T_0\subseteq T$
containing $y,x$, respectively. This is possible: if both original
sets have size $n$, they differ because $y$ blocks $W$ but not $T$;
a larger set offers at least two choices.
The three collections of missing $k$-sets through the designated
points have sizes $\binom n{k-1}$ and twice $\binom{n-1}{k-1}$.
An intersection involving the collection on $U$ has at most
$\binom{n-2}{k-2}$ members: a common generator must contain the
two distinct designated points. The other intersection has at most
$\binom{n-3}{k-2}$ members, since $W_0\ne T_0$.
Subtracting these three possible overlaps and applying Pascal's
identity gives
\[
 h\ge\binom n{k-1}+\binom{n-2}{k-1}+\binom{n-3}{k-1},
\]
contrary to \eqref{eq:h-criterion}. Thus two points suffice, and
Proposition~\ref{prop:two-blockers} proves the theorem.
\end{proof}

\begin{proof}[Proof of Theorem~\ref{cor:higher-linear}]
Put $r=k-1$. Since $h\le\binom{n+t-1}{r}-1$, it suffices to show
\begin{equation}\label{eq:whole-strip-comparison}
 \binom{n+t-1}{r}
 \le\binom nr+\binom{n-2}r+\binom{n-3}r.
\end{equation}
Divide by $\binom nr$. The product formula for binomial coefficients
and $n\ge(t+1)r$ give
\[
 \frac{\binom{n+t-1}r}{\binom nr}\le(1+1/t)^{t-1},\qquad
 \frac{\binom nr+\binom{n-2}r+\binom{n-3}r}{\binom nr}
 \ge1+(1-1/t)^2+(1-1/t)^3.
\]
For $t=4$ these bounds are $125/64$ and $127/64$. For $t\ge5$,
put $s=1/t$. The elementary logarithmic and exponential estimates give
\[
 (1+s)^{1/s-1}\le e^{1-4s/3}
 \le\frac{11}{4}-\frac{10}{3}s+\frac43s^2.
\]
For the first estimate use $\log(1+s)\le s-s^2/2+s^3/3$;
for the second use $e^{-x}\le1-x+x^2/2$ and $e<11/4$.
Comparison with $1+(1-s)^2+(1-s)^3$ leaves a cubic polynomial
which decreases on $[0,1/5]$ and has value $23/1500$ at $1/5$.

For $t=3$, retain the separate factors in the product estimates.
Since $r/n\le1/4$, $r/(n-1)\le2/7$, and $r/(n-2)\le1/3$,
the right side of \eqref{eq:whole-strip-comparison}, divided by
$\binom nr$, is at least $53/28$, whereas the left side is at most
$16/9$. Thus \eqref{eq:whole-strip-comparison} holds in every case,
and the theorem follows.
\end{proof}

\begin{remark}
For smaller $n$, Theorem~\ref{thm:h-criterion} still covers a
terminal interval of each strip, where fewer generators are missing.
For triples in the third strip, \eqref{eq:whole-strip-comparison}
holds for $n\ge7$. It suffices to check $n=7$, since each ratio
$\binom{n-j}{k-1}/\binom{n+2}{k-1}$, for $j=0,2,3$, is
nondecreasing in $n$.
\end{remark}

\subsection{Binomial values of the number of missing sets}

For the next bound the blocking points may vary freely. A set with a
chosen point has all but one of its immediate subsets containing that
point. A partial-shadow theorem therefore applies to the collections
of missing sets. It gives a continuous estimate, which is exact at
integer binomial values. For an integer \(s\ge1\) and a real
\(y\ge s-1\), put
\(\binom ys=y(y-1)\cdots(y-s+1)/s!\).
We also put \(\binom y0=1\) for \(y\ge0\).

We use the partial-shadow theorem proved by Chao and Yu
\cite[Theorem~1.7]{ChaoYu2023Joints}; see also \cite{ChaoYu2024}.
In the form needed here, if a family of \((\ell-1)\)-sets contains at
least \(\ell-1\) faces of each of \(\binom{x}{\ell-1}\) distinct
\(\ell\)-sets, where \(x\ge\ell-1\), then it has at least
\(\binom{x}{\ell-2}\) members.

\Needspace{12\baselineskip}
\begin{proposition}[A continuous bound for missing unions]\label{prop:continuous-missing}
Let $\mathcal B$ consist of all but $e$ of the $k$-sets of a finite
ground set, where $k\ge2$ and $e>0$. Define $y\ge k-1$ by
$e=\binom y{k-1}$. For every $u\ge k$, the number of $u$-subsets
which are not generated by $\mathcal B$ is at most
\begin{equation}\label{eq:continuous-missing}
 \begin{cases}
 \left\lfloor\binom y{u-1}\right\rfloor,&y\ge u-1,\\
 0,&y<u-1.
 \end{cases}
\end{equation}
For $e=0$ there are no such subsets.
\end{proposition}

\begin{proof}
For $u=k$ the number is $e$. Suppose $u>k$, and that there are
$L=\binom{x}{u-1}>0$ sets which are not generated, with $x\ge u-1$.
Choose a blocking point in each. For $k\le\ell\le u$, let
$\mathcal P_\ell$ consist of the $\ell$-subsets of these sets which
contain their chosen point. For $k<\ell\le u$, every member of
$\mathcal P_\ell$ has at least $\ell-1$ faces in $\mathcal P_{\ell-1}$.
Starting with $|\mathcal P_u|=L$, successive applications of the
partial-shadow theorem give
$|\mathcal P_k|\ge\binom{x}{k-1}$.
At each step, the real parameter defined by the actual size of
$\mathcal P_\ell$ is at least $x$, so the same lower bound on that
parameter is preserved.
All these $k$-sets are missing from $\mathcal B$. Hence
$\binom{x}{k-1}\le e$, so $x\le y$ and
$L\le\binom y{u-1}$. Integrality gives the first bound.
If $y<u-1$, even one blocking collection would contain more than $e$
missing $k$-sets. The case $e=0$ is immediate. This proves the proposition.
\end{proof}

\begin{theorem}[Binomial numbers of missing generators]\label{thm:binomial-missing}
Let $n>k\ge2$ and $t\ge1$. If
\[
 h=\binom a{k-1},\qquad k-1\le a\le m-2,
\]
for an integer $a$, then
\begin{equation}\label{eq:binomial-missing}
 |\langle\mathcal A\rangle_{\ge n}|
 \ge\sum_{u=n}^{m}\left[\binom mu-\binom a{u-1}\right]
 =U_k(N,n).
\end{equation}
Equality is attained simultaneously at every size $u\ge n$.
\end{theorem}

\begin{proof}
Apply Proposition~\ref{prop:continuous-missing} with $y=a$ and sum
over $u$. To obtain equality, choose $z\in V$ and
$X\subseteq V\setminus\{z\}$ of size $a$, and omit exactly the
$k$-sets $\{z\}\cup E$ with $E\in\binom X{k-1}$.
For $u>k$, the subsets which are not generated are precisely
\[
 \{z\}\cup T,\qquad T\in\binom X{u-1}.
\]
Sets avoiding $z$ are generated by the complete $k$-layer there.
If a set contains $z$ and a point outside $X\cup\{z\}$, a present
generator covers $z$, and its other points are covered by generators
avoiding $z$. This verifies the assertion.

Finally take $V=[m]$, $z=m$, and let $X$ be the last $a$ points of
$[m-1]$. The omitted $(k-1)$-sets are the last $\binom a{k-1}$ in
lexicographic order. The attaining family is therefore the max-lex
model, which proves the equality in \eqref{eq:binomial-missing}.
Since $a\le m-2$, its support is $V$. This proves the theorem.
\end{proof}

\subsection{The third strip}

A deletion gives more information than the number of missing generators
alone. Deleting a blocking point leaves a nearly complete family on one
fewer point; the generators through that point supply additional unions
to offset those which the remaining family fails to generate.
Here $t=3$, so $m=n+3$ and $q=n-k+3$. For this subsection put
\[
 Q(R)=\sum_{j=0}^{3}h_j^{(q)}(R).
\]
Thus $U_k(N,n)=S_2(n+2)+Q(R)$.

\begin{proposition}[Deleting a blocking point]\label{prop:third-deletion}
Suppose an $(n+1)$-subset is blocked by $z$, and let $d$ be the number
of generators through $z$. On $X=V\setminus\{z\}$, put
\[
 \mathcal B=\{A\in\mathcal A:z\notin A\},\qquad
 \mathcal G=\{A\setminus\{z\}:A\in\mathcal A,\ z\in A\}.
\]
Then precisely $e=d-R$ $k$-sets of $X$ are missing from $\mathcal B$,
and
\begin{equation}\label{eq:third-degree}
 R\le d\le\binom{n+2}{k-1}-\binom n{k-1}.
\end{equation}
If $c_u$ counts the $u$-subsets of $X$ not generated by $\mathcal B$,
then
\begin{equation}\label{eq:third-deletion}
 |\langle\mathcal A\rangle_{\ge n}|
 \ge S_2(n+2)+Q(d)-c_{n-1}-2(c_n+c_{n+1}+c_{n+2}).
\end{equation}
\end{proposition}

\begin{proof}
Since $|\mathcal B|=\binom{n+2}{k}+R-d$, we have $e=d-R\ge0$.
Every member of $\mathcal G$ meets the two points outside the blocked
$(n+1)$-set. This gives the upper bound on $d$.

There are $S_2(n+2)-c_n-c_{n+1}-c_{n+2}$ generated subsets of $X$
of size at least $n$. For unions through $z$, consider the subsets
$S\subseteq X$ of size at least $n-1$ containing a member of
$\mathcal G$. Complementation and Kruskal--Katona give at least
$Q(d)$ such candidates. If $S$ is generated by $\mathcal B$, adjoining
the corresponding generator through $z$ produces $S\cup\{z\}$.
At most $c_{n-1}+c_n+c_{n+1}+c_{n+2}$ candidates fail this condition.
Adding the two disjoint counts proves the proposition.
\end{proof}

A missing subset of $X$ of size at least $n$ is subtracted twice:
once from the unions avoiding $z$, and once from the candidates through
$z$. The latter subtraction is unnecessary when the generators through
$z$ already generate the candidate. We use this improvement at cutoff
four below.

\begin{theorem}[Triples in the third strip]\label{thm:third-small}
Let $n\ge3$ and let $\mathcal A$ be a family of $N$ triples with
support of size $n+3$, where
$\binom{n+2}{3}<N\le\binom{n+3}{3}$. Then
$|\langle\mathcal A\rangle_{\ge n}|\ge U_3(N,n)$.
\end{theorem}

\begin{proof}[Proof sketch]
The missing-generator criterion covers $n\ge7$. At $n=3$, apply
Theorem~\ref{thm:second-minimum} at cutoff four and add the $N$
triples. Thus only $n=4,5,6$ remains.

By Lemma~\ref{lem:bottom-level}, we may assume that a blocked
$(n+1)$-set exists. Delete a blocking point as in
Proposition~\ref{prop:third-deletion}. Let $L_{n,3}(e)$ be the upper
bound on $c_{n-1}+2(c_n+c_{n+1}+c_{n+2})$ obtained from
\eqref{eq:continuous-missing}, with $L_{n,3}(0)=0$. The required
comparison is
\begin{equation}\label{eq:third-arithmetic}
 Q(d)-Q(R)\ge L_{n,3}(d-R).
\end{equation}
Its left side measures the gain from the generators through $z$;
its right side bounds the loss after deletion.

For triples, Theorem~\ref{thm:h-criterion} leaves only $R\le8,7,4$
at $n=4,5,6$, respectively, while \eqref{eq:third-degree} gives
$d\le9,11,13$. At $n=5,6$, substitution in
\eqref{eq:third-arithmetic} proves the result. Appendix~\ref{app:third-strip}
gives the shadow counts and the elementary
bounds for $L_{n,3}$, separately for each cutoff.

At $n=4$, the family $\mathcal G$ is a graph on six points, all of
whose edges meet two fixed points $a,b$. There are nine possible
edges, and \eqref{eq:third-arithmetic} leaves only $d=5,8,9$.
These cases use the incidence of the edges. With five edges, a
common vertex gives the full star, which generates $26$ subsets
of size at least three. If there is no common vertex, the
complementary four-sets have singleton shadow of size six;
Kruskal--Katona then supplies one more candidate than $Q(5)$.
With eight edges, either $\mathcal G$ is the complete bipartite graph
$K_{2,4}$, or one of $a,b$
is adjacent to every other point. In the first case there are $37$
generated subsets of size at least three. In the second case the
full star gives $26$, and the three remaining edges give four more.
With nine edges, every such subset meeting $\{a,b\}$ is generated,
giving $37$. The corresponding unions through $z$ do not depend
on $\mathcal B$. Adding them to the old unions supplies the
remaining deficits. The short comparisons are given in
Appendix~\ref{app:third-strip}. This proves the theorem.

\end{proof}

\hypertarget{larger-supports}{%
\section{Larger supports}\label{larger-supports}\label{sec:larger-supports}}

The proofs of Theorems~\ref{thm:general-strips} and
\ref{thm:logarithmic-strips} use the same two alternatives. On a large
support, a minimal cover by generators gives enough distinct unions
simply by omitting a few of its members. On a smaller support, we first
delete points of small degree to obtain a dense family on at most
\(n+t\) points. We then restore some outside points using original
generators. Each point restored permits one more point to be omitted
from the remaining set while keeping the union above the cutoff.

Fix a family \(\mathcal A\) of \(N\) distinct \(k\)-sets in the
\(t\)-strip, where \(n>k\ge3\) and \(t\ge2\), and put
\[
 |\operatorname{supp}(\mathcal A)|=n+c,\qquad M=n+t,\qquad T=S_t(M).
\]
The max-lex count is at most \(T\). The minimum support, \(c=t\),
is covered by Sections~\ref{small-intersections} and
\ref{sec:higher-minimum} in both ranges below. For \(c>t\), we shall
prove the stronger bound \(|\langle\mathcal A\rangle_{\ge n}|\ge T\).

\subsection{Reducing the support}\label{subsec:reducing-support}

A \emph{generator cover} is a subfamily whose union is the full support.
We first allow ourselves to omit \(J\) members of such a cover,
where \(J\) is a positive integer.

\begin{lemma}\label{lem:large-support-cover}
If $J\ge1$ and $c\ge kJ$, then
\[
 |\langle\mathcal A\rangle_{\ge n}|
 \ge S_J(\lceil n/k\rceil+J).
\]
\end{lemma}
\begin{proof}
Choose an inclusion-minimal subfamily covering the support. It has
at least $\lceil(n+c)/k\rceil\ge\lceil n/k\rceil+J$ members.
Each has a point belonging to no other member of the cover; otherwise
it could be removed. Thus distinct subfamilies have distinct unions.
Omitting at most $J$ members removes at most $kJ$ points and leaves a union of
size at least $n+c-kJ\ge n$. Counting these choices proves the lemma.
\end{proof}

For smaller supports we use the degree of a point, the number of
current generators containing it. On a support \(Y\) of size \(s\),
suppose \(Y\setminus F\), with \(|F|\le t\), is not generated.
Some point \(x\) in it is uncovered by its internal generators, so
every generator through \(x\) meets \(F\). Hence
\(\deg(x)\le t\binom{s-2}{k-2}\). If every degree exceeds this
bound and \(s\ge M\), all omissions of at most \(t\) points are
generated, giving at least \(T\) large unions.

Otherwise delete a point of minimum degree and its generators. If
another point disappears, all its generators contained the chosen
point. Minimality of the degree implies that the two points had the
same generators. Hence all disappearing points lie in every deleted
generator, and at most \(k\) points disappear at one step.
For \(t<c<kJ\), write \(a=c-t\) and continue until the support has
at most \(M\) points. Unless the degree alternative has finished the
proof, we reach a set \(Z\) of size \(M-i\), where \(0\le i<k\),
after at most \(a\) deletions. Put
\begin{equation}\label{eq:common-missing-budget}
 \Lambda=at\binom{n+kJ-2}{k-2},\qquad
 E=\binom{M-1}{k-1}+\Lambda-1.
\end{equation}
At most \(\Lambda\) generators were lost. Consequently at most
\(E-\bigl(\binom Mk-\binom{M-i}k\bigr)\) generators are missing
on \(Z\). The subtraction records the smaller capacity of the
terminal support; a negative bound means that this support cannot occur.

\subsection{Recovering deleted points}

The first reduction controls the number of missing generators.
We now use that control to make every omission of at most \(t+d\)
points possible, where \(d\) is the number of extra omissions desired.
If an omission fails, all \(k\)-sets through an uncovered point in the
remaining set are missing. Delete that point and repeat. The missing
collections are disjoint, so the budget limits these further deletions.
The lemma then restores enough outside points; the parameter \(b\)
bounds how many must be covered.

\begin{lemma}\label{lem:recover-deleted-points}
Suppose the reduction above ends on $Z$. Let $1\le d\le a$ and
$b\ge d$ be integers, with $n>kb$. If
\begin{equation}\label{eq:recovery-budget}
 E<(b-d+1)\binom{n-b-1}{k-1},
\end{equation}
then
\begin{equation}\label{eq:recovery-count}
 |\langle\mathcal A\rangle_{\ge n}|\ge S_{t+d}(M+d-kb).
\end{equation}
\end{lemma}
\begin{proof}
\emph{Making the omissions possible.} Starting with $Y=Z$, delete an
uncovered point whenever an omission of at most $t+d$ points is not generated. After $j$ such deletions,
the next missing star has at least
$\binom{n-i-d-j-1}{k-1}$ members. These stars are disjoint: a later
one avoids every previously deleted point.

The initial deficit $i$ contributes to the same count. By Pascal's
identity, the amount $\binom Mk-\binom{M-i}k$ subtracted from $E$
is at least the sum of the first $i$ terms of the sequence starting at
$\binom{n-d-1}{k-1}$. After $w$ further deletions, therefore,
\[
 E\ge\sum_{\ell=0}^{i+w-1}\binom{n-d-\ell-1}{k-1}.
\]
If $i+w\ge b-d+1$, its first $b-d+1$ terms contradict
\eqref{eq:recovery-budget}. Thus the process stops with
$q:=d+i+w\le b$. The condition $n>kb$ ensures that the witness
sets have at least $k$ points up to the first forbidden deletion.
Every omission of at most $t+d$ points from the final $Y$ is generated.

\emph{Restoring points.} Choose $d+i$ original points outside $Z$ and
the $w$ points just deleted. There are enough outside points because $d\le a$. Cover
these $q$ points by original generators, and let $U$ be their union.
The cover uses at most $q(k-1)$ points of $Y$, so
\[
 |Y\cup U|\ge M+d,\qquad
 |Y\setminus U|\ge M+d-kq\ge M+d-kb.
\]
For each $F\subseteq Y\setminus U$ of size at most $t+d$, the union
$U\cup(Y\setminus F)$ is generated and has size at least $n$.
Different omissions give different unions. This proves the lemma.
\end{proof}

The point of the common budget is that an initial loss of support and
a later deletion have the same effect. We need only control their sum.
For the final comparison, splitting the binomial ratio into the old
\(t\) factors and the \(d\) new ones gives
\begin{equation}\label{eq:recovery-binomial-comparison}
 \frac{\binom{M+d-kb}{t+d}}{\binom Mt}
 \ge e^{-kbt/(n-kb+1)}
       \left(\frac{n-kb+1}{t+d}\right)^d.
\end{equation}
The exponential bounds the loss from the cover; the power is the gain
from the extra omissions. This follows directly from
\(\log(1-x)\ge-x/(1-x)\) for \(0\le x<1\). We now choose the parameters so that the
gain exceeds the loss.

\subsection{An explicit bound}

\begin{proof}[Proof of Theorem~\ref{thm:general-strips}]
We prove the slightly stronger range $n\ge(5/2)k^2t$. Minimum support
follows from Theorems~\ref{cor:second-minimum} and
\ref{cor:higher-linear}, so suppose $c>t$. A geometric-series estimate
gives $T<(11/10)\binom Mt$.

Take $J=2t$. The elementary comparison
\[
 \frac{\binom{\lceil n/k\rceil+2t}{2t}}{\binom Mt}
 \ge\left(\frac{n^2}{2k^2t(n+t)}\right)^t>\frac{11}{10}
\]
and Lemma~\ref{lem:large-support-cover} settle $c\ge2kt$.
For the remaining supports, use the reduction with $a=c-t$, and
take $d=\lceil a/k\rceil$. Use $b=2d+1$ for triples and $b=3d+2$
for $k\ge4$. These choices satisfy $d\le a$ and $n>kb$.

Two comparisons remain: the budget inequality
\eqref{eq:recovery-budget}, and a ratio greater than $11/10$ in
\eqref{eq:recovery-binomial-comparison}. The first limits the deletions;
the second makes the recovered count exceed $T$.

For the budget, divide by
$P=\binom{n-b-1}{k-1}$. For $k\ge4$, first remove the factor
$(k-1)/(n-b-k+1)$ from the ratio in the deletion term. Both
remaining binomial products are at most $e$: use $b\le7t$ and
$1+x\le e^x$ for $x\ge0$. The coefficient of the second product
is at most $a/(2k)$. Hence $E/P<3+2a/k\le2d+3$.
For triples, keeping the quadratic factors gives
$E/P<2+a/3\le d+2$.

For the recovered count, substitute the same choices in
\eqref{eq:recovery-binomial-comparison}. One may use
$n-kb\ge10t$ for triples, $n-kb\ge12t$ for $k=4$, and
$n-kb\ge kt((5/2)k-7)$ for $k\ge5$. Since $d\le2t$, these bounds
give the required ratio for triples and for $k\ge5$.
For $k=4,d\ge2$, the ratio
is at least $(4e^{-4/3})^d\ge16e^{-8/3}>11/10$.
For $k=4,d=1$, use $b=5$ and $n-20\ge30t$.
Lemma~\ref{lem:recover-deleted-points} now
gives more than $T$ large unions, completing the proof.
\end{proof}

\subsection{Allowing the set size to grow}

For large \(k\), the same argument can tolerate more deletions.
We let their number grow as a power \(k^\beta\), with \(\beta<1\),
so covering them still uses a vanishing fraction of the support.
We also increase \(J\) slightly above \(2t\); the extra omissions
make the direct cover count work at the smaller cutoff.

\begin{proof}[Proof of Theorem~\ref{thm:logarithmic-strips}]
It suffices to take $0<\varepsilon\le1$. Choose
$J=2t+\lceil\varepsilon t/2\rceil$. Since
$J\le(2+\varepsilon/2)t+1$, the assumed bound on $n$ gives
$k^2J/n\le\gamma\log k$ for some fixed $\gamma<1$, uniformly for
$t\ge2$. Fix $\gamma<\beta<1$. The first inequality lets the number of permitted
deletions grow faster than their estimated cost; the second keeps the
covering cost small compared with the support. All limits below are taken as
$k\to\infty$. The estimates are uniform in
$t$ and in larger $n$; their constants may depend on $\varepsilon$
and $\beta$.

Standard binomial estimates give, for a positive constant $C_\varepsilon$,
\[
 \frac{\binom{\lceil n/k\rceil+J}{J}}{\binom Mt}
 \ge\left(\frac{k^{\varepsilon/2}}
                 {C_\varepsilon(\log k)^2}\right)^t.
\]
The right side tends to infinity, whereas
$T/\binom Mt=1+O(t/n)=1+o(1)$. Thus the cover lemma settles
$c\ge kJ$ for sufficiently large $k$.

For $t<c<kJ$, reduce the support and put
$d=\lceil(c-t)/k\rceil$, $b=d+\lceil k^\beta d\rceil-1$.
Here $d\le J\le3t$, and $kb/n=O(k^{\beta-1}\log k)=o(1)$.
Writing $P=\binom{n-b-1}{k-1}$, comparison of binomial products gives
\[
 \frac EP\le2+O(dk^\gamma\log k)
       =o(k^\beta d)<b-d+1.
\]
To see the exponent, the product contributed by the deletion loss
is at most $\exp(k^2J/n+o(1))=O(k^\gamma)$; the contribution from
$b$ is $o(1)$ because $\beta<1$. The factor preceding that product
is $O(d\log k)$. These observations also give uniformity of the
estimates. The recovery lemma therefore applies.

Finally, the logarithm of the right side of
\eqref{eq:recovery-binomial-comparison} is
\[
 d\bigl(2\log k-\log\log k-O(1)\bigr)
       -O(dk^{\beta-1}\log k),
\]
or larger. It tends to infinity uniformly. The covering cost is
$o(d)$, while each extra omission contributes a factor of order $n/t$. Hence the recovered count exceeds
$T$.

For large $k$, the assumed cutoff also implies
$n\ge(t+1)(k-1)$, uniformly for $t\ge2$, so the minimum-support
case follows from Theorems~\ref{cor:second-minimum} and
\ref{cor:higher-linear}. Apply the first assertion with
$\varepsilon/2$ and use $1/t\le\varepsilon/2$ to obtain the last
assertion. This completes the proof.
\end{proof}

The additional \(1\) in the numerator allows \(J>2t\) even when
\(t\) is fixed. When \(t\) is large, its contribution can be absorbed
into \(\varepsilon\), as in the last assertion of the theorem.

\section{Induction on the support}\label{sec:support-induction}

The minimum-support theorems give the initial cases of a possible
induction. They do not give its step. Deleting or identifying points
can reduce the number of distinct generators, and the model bound may
drop sharply when this happens. To prove the bound for the original family, we must account for this
decrease by counting unions beyond those supplied by the reduced family.

We first formulate a sufficient conjecture for an induction covering
all strips. We then prove estimates for deletion which will give an
unconditional argument for triples in Section~\ref{sec:second-small}.
For a finite generator family $\mathcal A$ and an integer cutoff $h$, write
\[
 M_h(\mathcal A)=|\langle\mathcal A\rangle_{\ge h}|,
 \qquad U_k(0,h)=0.
\]
Here $k\ge1$, and the convention for $U_k(0,h)$ covers an empty
generator family. For a nonnegative integer $a$ and $h<k$, the definition
gives $U_k(a,h)=U_k(a,k)$: every generated union then meets the cutoff.

\subsection{Identifying two points}

Let $\mathcal A$ consist of $N$ distinct $k$-sets, $k\ge2$, on a
support $V$. For distinct $x,y\in V$, let
$\pi:V\longrightarrow V\setminus\{y\}$ send $y$ to $x$ and fix
every other point. Apply $\pi$ to sets and families by taking their images.
The image
$\mathcal I=\pi(\langle\mathcal A\rangle)$ is union-closed.
Some image generators have size $k-1$, and others coincide. However,
unions of the original generators may supply further $k$-sets in
$\mathcal I$. Thus we use its entire $k$-th level.

\begin{lemma}[Recovering image generators]\label{lem:ind-image}
Let $d_{xy}$ count generators containing both $x,y$, and put
\[
 \mathcal L_x=\{A\setminus\{x\}:A\in\mathcal A,\ x\in A,\ y\notin A\},
 \qquad
 \mathcal L_y=\{A\setminus\{y\}:A\in\mathcal A,\ y\in A,\ x\notin A\}.
\]
Let $r_{xy}$ count $(k+1)$-sets $W$ containing $x,y$ and at least
two generators, for which neither $W\setminus\{x\}$ nor
$W\setminus\{y\}$ is a generator. Then
\begin{equation}\label{eq:ind-image-size}
 |\mathcal I_k|
 =N-d_{xy}-|\mathcal L_x\cap\mathcal L_y|+r_{xy}.
\end{equation}
For a cutoff $n\ge k$, let $b_{xy}$ count generated $n$-sets
containing $x,y$, and let
\[
 e_{xy}=\sum_{Y\in\mathcal I_{\ge n}}
 \bigl(|\{W\in\langle\mathcal A\rangle:\pi(W)=Y\}|-1\bigr).
\]
Then
\begin{equation}\label{eq:ind-fibres}
 M_n(\mathcal A)=|\mathcal I_{\ge n}|+b_{xy}+e_{xy}.
\end{equation}
\end{lemma}
\begin{proof}
There are $N-d_{xy}-|\mathcal L_x\cap\mathcal L_y|$ distinct
$k$-set images of generators. Every other image $k$-set comes from
a unique $(k+1)$-set through $x,y$. Such a set is generated exactly
when it contains at least two generators. Its image has already
been counted exactly when one of its two indicated deletions is a
generator. This proves \eqref{eq:ind-image-size}.

Identification decreases a set's size by at most one. Among unions
meeting the cutoff, precisely the $n$-sets through both identified
points fall below it. All other losses in the count are coincidences
of images, counted by $e_{xy}$. This proves \eqref{eq:ind-fibres}.
\end{proof}

Put $L_{xy}=\min\{N,|\mathcal I_k|\}$. If the conjectured bound is
known on smaller supports, select $L_{xy}$ members of $\mathcal I_k$.
Their closure is contained in $\mathcal I$, so the lemma gives
\begin{equation}\label{eq:ind-payment}
 M_n(\mathcal A)\ge U_k(L_{xy},n)+b_{xy}+e_{xy}.
\end{equation}
In particular, retaining $N$ image generators finishes the step.
When generators are lost, the following is a sufficient local
conjecture.

\begin{conjecture}[Compensation under identification]\label{conj:ind-compensation}
Suppose $n\ge k\ge2$, $N>\binom nk$, and $\mathcal A$ has
$N$ distinct $k$-sets with support larger than $\sigma_k(N)$.
There are distinct support points $x,y$ for which
\begin{equation}\label{eq:ind-compensation}
 b_{xy}+e_{xy}\ge U_k(N,n)-U_k(L_{xy},n).
\end{equation}
\end{conjecture}

This is the precise comparison needed in \eqref{eq:ind-payment}
when we use only the model bound for the image family. It is an
additional conjecture; the minimum-support results do not establish it.
It may demand more than is needed for a particular family: a surplus
above the model bound in the image would also contribute to the induction.

\begin{theorem}[Conditional support induction]\label{thm:ind-conditional}
Fix $k\ge2$. Suppose the layered conjecture holds on minimum
support in every strip, and suppose
Conjecture~\ref{conj:ind-compensation} holds for this $k$.
Then the layered conjecture holds on arbitrary support in every strip.
\end{theorem}
\begin{proof}
Induct on the support size, simultaneously for all numbers of
generators and all cutoffs. If $N\le\binom nk$, the model count is
zero or one, and the bound follows from the full support. Otherwise,
minimum support is an assumed initial case. On a larger support,
choose $x,y$ as in Conjecture~\ref{conj:ind-compensation}. The selected
$L_{xy}$ image generators have smaller support. Apply induction
in \eqref{eq:ind-payment} and then \eqref{eq:ind-compensation}.
Their sum is $U_k(N,n)$, proving the step.
\end{proof}

The induction is on support, rather than strip number: the number of
generators remaining after an identification can belong to a different
strip. At the endpoint $N=\binom rk$, minimum support forces the
complete $k$-th level on $r$ points, so that initial case is immediate.
The corresponding assertion on larger support still requires an
inductive step.

\subsection{Deleting and restoring a point}

There is another way to retain information after a deletion. Let $p$
be a support point, let $\mathcal B$ be the generators avoiding it,
and let
$\mathcal H=\{A\setminus\{p\}:p\in A\in\mathcal A\}$ be its
\emph{link}. Put $a_j=M_j(\mathcal B)$. Let $I$ count the unions
of $\mathcal B$ of size at least $n-1$ which contain no member of
$\mathcal H$, and let $G$ count the sets of size at least $n-1$
in $\langle\mathcal B\cup\mathcal H\rangle
 \setminus\langle\mathcal B\rangle$. Then
\begin{equation}\label{eq:ind-adjacent-exact}
 M_n(\mathcal A)=a_n+a_{n-1}+G-I.
\end{equation}
Indeed, a union through $p$, with $p$ removed, uses at least one
link member. An old union has such a representation precisely when
it contains a link member, and every new mixed union uses one.

Thus deletion has two effects: some old unions cannot be extended
through $p$, while the link creates new mixed unions. Identity
\eqref{eq:ind-adjacent-exact} keeps both effects, for every $k$ and
every strip. If bounds
$a_n\ge\eta_0$ and $a_{n-1}\ge\eta_1$ are available, the exact
remaining comparison is
\[
 (a_n-\eta_0)+(a_{n-1}-\eta_1)+G-I
 \ge U_k(N,n)-\eta_0-\eta_1.
\]
Both surpluses matter. A proof that controls only the new unions
can fail even when the lower-cutoff count already exceeds the target.

To carry a level across the deletion, we seek an injection sending
each old $(n-1)$-union $U$ to $U\cup H$ for some link member $H$.
Here is one general condition that supplies such an injection.

\begin{lemma}[Disjoint link members]\label{lem:ind-disjoint}
Let $\mathcal H$ contain $\binom{s}{\lfloor s/2\rfloor}$ pairwise
disjoint $s$-sets, where $s\ge1$. For every family $\mathcal U$ of
sets of one size, there is an injection
$\phi:\mathcal U\longrightarrow\{U\cup H:U\in\mathcal U,\ H\in\mathcal H\}$ with
$\phi(U)=U\cup H$ for some $H\in\mathcal H$.
\end{lemma}
\begin{proof}
Use just the indicated disjoint members, and put
$b=\binom{s}{\lfloor s/2\rfloor}$. Assign to itself any $U$ that
contains one of them. Each remaining $U$ has $b$ distinct possible
images. Such an image contains exactly one of the chosen members:
if it contained a second, that disjoint member would already be in
$U$. An image of size $|U|+j$ therefore has at most
$\binom sj\le b$ predecessors. The bipartite graph of possible
images has left degree $b$ and right degree at most $b$.
Hall's condition follows by counting edges, and a matching proves
the lemma.
\end{proof}

For triples we can dispense with disjointness.

\begin{lemma}[Two edges]\label{lem:ind-two-edge}
Let $e,f$ be distinct two-element sets, and let $\mathcal U$ be
a family of sets of the same size. There is an injection
$\phi:\mathcal U\longrightarrow\{U\cup e,U\cup f:U\in\mathcal U\}$ such that
$\phi(U)\in\{U\cup e,U\cup f\}$ for every $U\in\mathcal U$.
\end{lemma}
\begin{proof}
Write $r$ for the common size. Send a set containing either edge
to itself. For the remaining sets form a bipartite graph with
possible images $U\cup e,U\cup f$ on the right.

Every right vertex has degree at most two. An image of size $r+2$
has only the predecessors obtained by removing $e$ or $f$. An
image $W$ of size $r+1$ has a predecessor $W\setminus\{v\}$ only
when $v$ belongs to every one of $e,f$ contained in $W$; otherwise
that predecessor was assigned to itself. There are again at most
two choices. A left vertex has degree one only when its two images
coincide. The added point is then the common endpoint of the two
edges, and the right vertex also has degree one.

Remove these isolated edges. Every remaining left vertex has degree
two, and every right vertex has degree at most two. Counting edges
proves Hall's condition. A matching gives the injection, with images
larger than $r$, disjoint from those chosen at the start.
\end{proof}

\begin{corollary}[Raising the cutoff]\label{cor:ind-bridge}
Let $n\ge1$ be an integer and let $\mathcal A$ be a triple family
with support of size $m\ge n+3$. Suppose $p$ belongs to at least two
generators. If $\mathcal B$
consists of the generators avoiding $p$, then
\begin{equation}\label{eq:ind-bridge}
 M_n(\mathcal A)\ge M_{n-1}(\mathcal B)+1.
\end{equation}
The last term may be replaced by the number of unions through $p$
of size greater than $n+2$, which is at least
$\lceil(m-n-2)/3\rceil$.
\end{corollary}
\begin{proof}
Keep the unions of $\mathcal B$ of size at least $n$. Apply the
lemma to its $(n-1)$-unions, using two link edges, and restore $p$.
These images contain $p$ and have sizes between $n$ and $n+2$.
The full support supplies one further union. Starting with a
generator through $p$ and adjoining generators until the support
is covered gives the stronger assertion: each step adds at most
three points.
\end{proof}

The same proof works for $k$-sets whenever the link admits the
required injection. Its images have size at most $n+k-1$, so every
union through $p$ of size at least $n+k$ is additional.
Lemma~\ref{lem:ind-disjoint} gives one sufficient condition for
general $k$. For triples, two distinct link members always suffice.

\subsection{How many generators remain}

The following estimate does not depend on the strip or on the
special form of a triple link.

\begin{lemma}[Two deletions]\label{lem:ind-retention}
Let $\mathcal A$ have $N$ distinct $k$-sets and minimum degree
$\delta$. Delete a point of minimum degree, leaving a family with support
of size $w$, and then delete a point of minimum degree in that family.
Let $a,b$ be real numbers satisfying
$\delta\le w+b$, $w\ge a>k$, $b\ge0$, and $N\ge a+b$.
The number $q$ of surviving generators satisfies
\begin{equation}\label{eq:ind-retention}
 q\ge\left\lceil(N-a-b)\left(1-\frac{k}{a}\right)\right\rceil.
\end{equation}
\end{lemma}
\begin{proof}
Averaging degrees after the first deletion gives
$q\ge(N-\delta)(1-k/w)$. If $\delta\le a+b$, use $w\ge a$
to obtain the claim. For larger $\delta$, the hypothesis forces a
larger surviving support, which improves the averaging bound.
Put $u=\delta-b>a$, so that $w\ge u$.
If $m$ was the original support size, minimum degree gives
$kN\ge m\delta\ge u(u+b)$. Since $u>a>k$,
this implies $k(N-b)\ge au$. Hence
\[
 (N-b-u)(1-k/u)-(N-b-a)(1-k/a)
 =(u-a)\left(\frac{k(N-b)}{au}-1\right)\ge0.
\]
Replace $w$ by $u$ in the averaging bound and take the ceiling.
\end{proof}

\subsection{Traces supply the additional unions}

The second deletion separates two collections of unions. The
surviving generators give those avoiding the deleted point.
To find unions containing it, adjoin one fixed generator. Their
distinct parts outside that generator are what we must count.

\begin{lemma}[Trace capacity]\label{lem:ind-trace}
Let $\mathcal C$ be a family of $k$-sets avoiding a point $x$,
and let $E$ be a $k$-set containing $x$. Let $\mathcal D$ be
the union-closure of the traces $C\setminus E$, $C\in\mathcal C$,
including the empty set, and put
$f_j=|\mathcal D_j|$, $Q_h=|\mathcal D_{\ge h}|$ for integer indices.
In particular, $f_0=1$ and $Q_h=|\mathcal D|$ for $h\le0$. Then
\begin{equation}\label{eq:ind-capacity}
 |\mathcal C|\le\sum_{j=1}^k\binom{k-1}{k-j}f_j.
\end{equation}
If $\mathcal B$ contains $\mathcal C$ and $E$, then
\begin{equation}\label{eq:ind-trace-unions}
 M_h(\mathcal B)\ge M_h(\mathcal C)+Q_{h-k}.
\end{equation}
In particular, if $v$ is a nonnegative integer and
$f_j\le\binom vj$ for $1\le j\le k$, then
\begin{equation}\label{eq:ind-vandermonde}
 |\mathcal C|\le\binom{v+k-1}{k}.
\end{equation}
\end{lemma}
\begin{proof}
A trace of size $j$ has at most $\binom{k-1}{k-j}$ preimages:
the remaining points lie in $E\setminus\{x\}$. The traces are
nonempty, since a $k$-set avoiding $x$ cannot be contained in $E$.
This proves \eqref{eq:ind-capacity}. For each trace union $D$,
the set $E\cup D$ is generated, has size $k+|D|$, and contains
$x$. Distinct traces give distinct unions, disjoint from those
generated by $\mathcal C$. This proves \eqref{eq:ind-trace-unions}.
Finally, Vandermonde's identity gives
\[
 \sum_{j=1}^k\binom{k-1}{k-j}\binom vj
 =\binom{v+k-1}{k},
\]
the omitted $j=0$ term being zero.
\end{proof}

Here is how lower-cutoff theorems enter. The $j$-th level of
$\mathcal D$ generates a subfamily of $\mathcal D$. If a known
bound says that $c_j+1$ distinct $j$-sets force more than $Q$
unions of size at least $h$, where $c_j,Q$ are nonnegative integers,
then $Q_h\le Q$ implies $f_j\le c_j$.
Substitution in \eqref{eq:ind-capacity} bounds the number of
surviving generators. In particular, the bounds
$f_j\le\binom{n-k}{j}$ give
$|\mathcal C|\le\binom{n-1}{k}$. Retaining more generators forces
the additional trace unions needed for induction.

All these trace and retention estimates are uniform in $k$ and $t$.
For triples, the two-edge injection applies to every link with at
least two members. For larger generators, a link need not contain the
disjoint members required by Lemma~\ref{lem:ind-disjoint}; a general
lifting argument is still needed. Section~\ref{sec:second-small}
combines the two-edge injection with the trace estimate to prove the
triple case, after establishing the small cutoffs by elementary counts.

\section{The second strip for small generators}\label{sec:second-small}

We now prove Theorem~\ref{thm:small-k}. For triples, the first
deletion lowers the cutoff. A second deletion leaves enough triples
to force further unions through their traces. We induct on the cutoff, and at each cutoff on the actual support.
The minimum-support theorem supplies the initial support. The first-strip
theorem and elementary counts supply the cutoffs through six. Thereafter the two deletions
give a common inductive step. The final subsection treats four-sets
and five-sets by the support reduction of Section~\ref{sec:larger-supports}.

For the second-strip target, write
\[
 N=\binom{n+1}{k}+R,\qquad
 T_{n,k}(R)=U_k(N,n),\qquad
 T_{\max}(n)=S_2(n+2)=1+\binom{n+3}{2}.
\]
Here $n\ge k$ and $1\le R\le\binom{n+1}{k-1}$ are integers,
and $T_{n,k}(R)\le T_{\max}(n)$.
The minimum-support case is Theorem~\ref{thm:second-minimum}.

We first record a degree observation used throughout the section. The
degree of a point is the number of generators containing it; its
\emph{link} consists of those generators with the point removed.
For a generator family $\mathcal A$, write $m$ for its support-size
and $\delta$ for its minimum degree.
On a support of size $m\ge k+3$, put
\begin{equation}\label{eq:second-degree-threshold}
 \lambda(m,k)=\binom{m-2}{k-2}+\binom{m-4}{k-3}.
\end{equation}
If every degree exceeds $\lambda(m,k)$, then at least
$1+\binom m2$ unions have size at least $m-2$. Indeed, a point $x$
outside a deleted set $D$ is uncovered precisely when $D$ meets every
member of its link. A singleton can meet every link member only if
the degree is at most $\binom{m-2}{k-2}$. Two distinct pairs can both
meet every member only if the degree is at most $\lambda(m,k)$.
For intersecting pairs, a link member must contain their common
point or both other points; disjoint pairs give a smaller bound.
Consequently all sets of size $m-1$ are generated, and each point
witnesses the failure of at most one set of size $m-2$. Including
the full support, the number of generated sets on these three levels
is at least $1+m+\binom m2-m$, as asserted.
In particular, on support at least $n+3$ it gives $T_{\max}(n)$ large
unions. For triples, $\lambda(m,3)=m-1$.

\subsection{Triples: the model and the degree reduction}

For triples write $T_n(R)=T_{n,3}(R)$ and $H(v)=S_2(v)$.
The model contains all triples on $[n+1]$ and the triples
$\{n+2\}\cup e$, where $e$ runs through the first $R$ pairs
of $[n+1]$ in lexicographic order. Counting the sets containing
a link edge gives
\begin{equation}\label{ss3-eq:target}
\begin{array}{c|cccccc}
 R&1&2&3\le R\le n&n+1&n+2\le R\le2n-1&R\ge2n\\ \hline
 H(n+2)-T_n(R)&2n+1&n+2&n+1&2&1&0.
\end{array}
\end{equation}
Indeed, a set through $n+2$ fails to be generated exactly when its
remaining points contain no link edge. Thus the subtracted quantities
count independent sets of sizes at least $n-1$ in the initial lexicographic
graph. In particular,
\begin{equation}\label{ss3-eq:lower-steps}
 T_n(1)=\frac{n^2+n+6}{2},\quad T_n(2)=H(n+1),\quad
 T_n(R)=H(n+1)+1\quad(3\le R\le n).
\end{equation}
The count never exceeds $H(n+2)$ and reaches that value at $R=2n$.

Besides the minimum-support theorem, we use the first-strip theorem
on arbitrary support. Its small values for triples are
\begin{equation}\label{ss3-eq:first-small}
\begin{array}{c|ccc}
 |\mathcal A|&5&6,7&8,9,10\\ \hline
 M_{4}(\mathcal A)\text{ is at least}&4&5&6.
\end{array}
\end{equation}
Taking a subfamily allows us to use any of these bounds when more
generators are present. We also use the theorem for pairs of
Leck--Roberts--Simpson \cite[Theorem~3]{LeckRobertsSimpson2012}. Its threshold-weight specialization says
that, for integers $v\ge1$ and $0\le b<v$, a family of
$\binom v2+b$ pairs has, at an integer cutoff $h\ge2$, at least
\begin{equation}\label{ss3-eq:pairs}
 \sum_{j\ge h}\left[\binom{v+1}{j}-\binom{v-b}{j-1}\right]
\end{equation}
unions of size at least $h$. This is the count for a complete graph
on $v$ points and $b$ edges from one new point. 

For $n=3$, every union is a triple or has size at least four. Hence
the first-strip theorem at cutoff four proves Theorem~\ref{thm:small-k}(i)
immediately. We assume $n\ge4$ below.

\begin{lemma}\label{ss3-lem:degree}
Let $\mathcal A$ have $N$ triples and support of size $m\ge7$, and let $p$
have minimum degree $\delta\ge2$. Let $\mathcal B$ be the generators avoiding
$p$, with support of size $w$.
If $\delta\le m-1$, then
\begin{equation}\label{ss3-eq:degree}
 w\in\{m-1,m-2\},\qquad \delta\le w,\qquad
 \delta\le d(N):=\left\lfloor\frac{\sqrt{12N+1}-1}{2}\right\rfloor.
\end{equation}
If instead $\delta\ge m$, then
\begin{equation}\label{ss3-eq:high-degree}
 M_{m-2}(\mathcal A)\ge1+\binom m2.
\end{equation}
\end{lemma}
\begin{proof}
Any other point disappearing when $p$ is deleted has its incident
family contained in the incident family of $p$. Minimum degree forces
equality. Thus all disappearing points occur together. Three such
points would force degree one. If two points disappear, every removed
triple consists of that pair and one other point, so $\delta\le m-2$.
This proves the support assertions. Also
$\delta(\delta+1)\le m\delta\le3N$, proving the bound on $\delta$.

The last assertion is the degree observation above, since $\lambda(m,3)=m-1$.
\end{proof}

In particular, if $m\ge n+3$, the high-degree bound is at least
$H(n+2)$. Such a family already satisfies the theorem.

\begin{lemma}\label{ss3-lem:one}
In proving Theorem~\ref{thm:small-k}(i) at a fixed cutoff $n\ge4$ by induction
on support, it is enough to consider minimum degree at least two.
\end{lemma}
\begin{proof}
Suppose $p$ occurs only in $\{p,a,b\}$, and let $m\ge n+3$ be the
support size. If $\{z,a,b\}\notin\mathcal A$ for some support point
$z\notin\{p,a,b\}$, identify $p$ with $z$. All image generators are
distinct triples, the support decreases by one, and the number of large
unions cannot increase. The support induction applies.

Otherwise $\mathcal A$ contains every triple $\{a,b,z\}$ with
$z\notin\{a,b\}$. Choose $n+1$ such points $z$. Unions of at least
$n-2$ of these triples give
\[
 1+(n+1)+\binom{n+1}{2}+\binom{n+1}{3}\ge H(n+2)
\]
large unions. The difference is $\binom{n+1}{3}-(n+2)\ge0$ for
$n\ge4$.
\end{proof}

\subsection{Small families of triples}

We begin with four generators. This count will be used in the
support induction for the next lemma.

\begin{lemma}\label{ss3-lem:four}
Four distinct triples with support of size at least six have at least
eight nonempty unions.
\end{lemma}
\begin{proof}
Suppose first that two triples $A,B$ are disjoint. If another triple
$C$ is not contained in $A\cup B$, then
$A\cup B,A\cup C,B\cup C,A\cup B\cup C$ are four distinct unions
larger than a triple. Otherwise all four triples lie in $A\cup B$.
For a third triple $C$, the three sets $A\cup C,B\cup C,A\cup B$
are distinct. A fourth triple $D$ must give a further union: if it
did not, $A\cup D=A\cup C$ and $B\cup D=B\cup C$, forcing $D=C$.

If no pair is disjoint, all pair unions have size at most five. There
are at least three distinct pair unions. For otherwise, let
$P=A\cup B$, and choose a triple $C$ with a point outside $P$.
Both $A\cup C$ and $B\cup C$ must be the same other union $Q$,
so $P\subseteq Q$. No triple can then have a point outside $Q$,
since its union with $A$ would be a third pair union. This contradicts
support at least six. The three pair unions and the full support
are four distinct unions larger than a triple.
\end{proof}

The next lemma measures the extra unions forced by support of size
at least six. We use one simple injection several times: adjoining a
fixed point to distinct triples gives distinct sets. An image of size
three is unchanged; from an image of size four we recover the triple
by deleting the fixed point.

\begin{lemma}\label{ss3-lem:gap}
If a family $\mathcal A$ of $7\le a\le10$ distinct triples has support
of size at least six, then it generates at least $a+2$ sets of size
at least four.
\end{lemma}
\begin{proof}
We induct on support. On six points, either every five-subset is
generated, or a blocking point forces the generators into three simple
forms. Once that initial case is proved, the degree reduction and the
cutoff-raising lemma handle larger supports.

\emph{Six points, with every five-subset generated.} Let the support
be $V$, and let $g$ be the number of generated four-subsets.
Each triple lies in three four-subsets. A four-subset which is not
generated contains at most one triple, while any four-subset contains
at most four. Hence
\[
 3a\le15+3g.
\]
The six five-subsets, the full support, and these $g\ge a-5$
four-subsets give the assertion.

\emph{Six points, with a missing five-subset.} Suppose
$V\setminus\{x\}$ is not generated, and choose an uncovered point $y$. Every generator through $y$ also contains
$x$. Put $P=V\setminus\{x,y\}$, a four-set. The generators have
three forms:
\[
 \alpha\text{ triples on }P;\qquad
 \beta\text{ sets }\{x\}\cup e\ (e\in G);\qquad
 \gamma\text{ sets }\{x,y,z\}\ (z\in C).
\]
Here $G$ is a graph on $P$ with $\beta$ edges, $C\subseteq P$,
$|C|=\gamma\ge1$, and
$\alpha+\beta+\gamma=a$.

We separate the count according to the number $\alpha$ of triples
avoiding $x$. With at most one such triple we count in the link of $x$;
with at least two, their union is $P$ and we count the unions through
$x$ and through $\{x,y\}$ separately.

If $\alpha=0$, the link of $x$ has at least seven edges, and
\eqref{ss3-eq:pairs} gives at least twelve unions of size at least three
in that link. Restoring $x$ proves the claim. The same argument works
when $\alpha=1$ and $a\ge8$.

For $\alpha=1,a=7$, the link of $x$ has six edges. If its support
has five points and no leaf (a vertex of degree one), all five
four-subsets are generated. Writing $g_3$ for the number of triples
generated by the link, the incidence count $18\le10+2g_3$ gives
$g_3\ge4$, hence at least ten unions of size at least three.
If the graph has a leaf, its other four points form the complete graph
$K_4$ minus an
edge. There are four generated triples on those points and at least
two through the leaf; the core four-set, three four-sets through the
leaf, and the full support give at least eleven unions.
If instead the link has four points, it is $K_4$. The remaining
triple must be $\{z,u,v\}$, with $z$ outside the link and $u,v$
inside it. The link supplies five large unions through $x$ avoiding
$z$. Writing its other two points as $w,t$, adjoining the generators
with link edges $uv,uw,ut,wt$ to $\{z,u,v\}$ gives four further
unions. This proves the required bound nine.

It remains to consider $\alpha\ge2$. In this case $P$ itself is
generated. If $\beta=0$, then $\alpha+\gamma\ge7$ implies either
$\gamma=4$, or $\gamma=3,\alpha=4$. In the first case, the singletons
of $C$ give eleven large unions through $\{x,y\}$, in addition to
$P$. In the second, they give the three pairs of $C$; the four given
triples on $P$ and $P$ itself give five more unions through
$\{x,y\}$. Together with $P$ this gives nine, as required.

Assume $\beta>0$. Let $u$ count three-subsets of $P$ that can be
generated using an edge of $G$, together with the given triples on
$P$. Let $v_i$, for $i=2,3$, count $i$-subsets of $P$ that can be
generated using at least one singleton of $C$, together with those
edges and triples. The sets counted by $u$ yield unions containing
$x$ and avoiding $y$; those counted by $v_i$ yield unions containing
both. Also $P$, $P\cup\{x\}$, and $V$ are generated. Thus
\begin{equation}\label{ss3-eq:four-count}
 M_{4}(\mathcal A)\ge3+u+v_2+v_3.
\end{equation}

The remaining distinction is whether $C$ is a singleton. If
$\gamma\ge2$, every given triple on $P$ meets
$C$, so $v_3\ge\alpha$, and
\[
 v_2\ge\max\left\{\binom\gamma2,\,
                \beta-\binom{4-\gamma}{2}\right\}.
\]
For $\gamma=2$ we have $u\ge2$. To see this, a graph with at least
three edges either has a three-edge star, which supplies three
triples, or every three-subset contains an edge, so the given
$\alpha\ge2$ triples count. With two edges, $a\ge7$ forces
$\alpha\ge3$, and at most one three-subset contains no edge.
With one edge, $\alpha=4$, and its two containing triples count.
Thus $u+v_2\ge\beta+1$. For $\gamma=3$, the same argument gives
$u\ge2$ when $\beta\ge3$; two edges give $u\ge1$; and one edge
needs no contribution from $u$. Since $v_2\ge\max\{3,\beta\}$,
we obtain $u+v_2\ge\beta+2$. These inequalities and
\eqref{ss3-eq:four-count} prove the claim. If $\gamma=4$, simply use
$v_2=6,v_3=4$.

If $C=\{z\}$, let $d$ be the degree of $z$ in $G$.
Then $v_2=d$ and $v_3\ge\beta-d$: adjoin $z$ to each edge avoiding
it. Usually $u\ge\alpha$, which finishes the count. Here is the
complete description of the exceptions. Since
$\alpha+\beta\ge6$ and $\alpha\le4$, we have $\beta\ge2$.
Four edges meet every three-subset. Three edges do so unless they
form a star, which still generates three three-subsets. With two
edges, $\alpha=4$; disjoint edges meet every three-subset, while
adjacent edges miss just one. Thus $u<\alpha$ occurs only when
$\alpha=4$ and $G$ is either two adjacent edges or a three-edge
star. In either case $u=3$ and $v_3=3$, since all four triples on
$P$ are given. For two edges \eqref{ss3-eq:four-count} gives at least
nine; for the star, $d\ge1$ gives at least ten. These are exactly
the required bounds. This completes the six-point proof.

\emph{The induction step on support.} Let $m\ge7$. A degree-one point is
identified as in Lemma~\ref{ss3-lem:one}, keeping support at least six;
if no identification is possible, five triples consisting of a fixed
pair and distinct third points give
$2^5-5-1=26$ large unions. Assume then that minimum degree is at
least two. For $a=8,9,10$, averaging gives minimum degree at most
$3,3,4$, respectively. Deletion leaves at least $5,6,6$ triples.
By \eqref{ss3-eq:first-small}, their total union-closures have at least
$9,11,11$ members. Corollary~\ref{cor:ind-bridge} at cutoff four adds
one and proves the claim.

For $a=7$, minimum degree two is handled in the same way. The only
remaining case has $m=7$ and every degree three. Deletion leaves four
triples on either six or five points. On six points these have at
least eight nonempty unions, by Lemma~\ref{ss3-lem:four}, and
Corollary~\ref{cor:ind-bridge} gives nine. On five points, two points
$p,y$ disappeared together. The removed triples are $\{p,y,z\}$
for three distinct points $z$ of the surviving support $Z$.
Through $\{p,y\}$ we obtain the three pairs of these points, the
four distinct sets $B\cup\{z_0\}$ for a fixed such point $z_0$
and a surviving triple $B$, and $Z$ itself. Their outside sizes are
$2$, $3$ or $4$, and $5$, respectively. They give eight large
unions through the pair, and $Z$ gives one avoiding it.
\end{proof}

\subsection{The cutoffs four and five}

For a fixed cutoff we need only check the first value of $N$ at each
jump of \eqref{ss3-eq:target}. A larger family contains a subfamily at
that value, with no more unions. In each case we induct on support.
The minimum-support case is Theorem~\ref{thm:second-minimum}. Lemmas~\ref{ss3-lem:degree}
and~\ref{ss3-lem:one} allow us, on larger support, to assume
$2\le\delta\le m-1$ and delete a minimum-degree point. Write
$\mathcal B$ for the surviving family and $b=|\mathcal B|$.

\begin{proposition}\label{ss3-prop:four}
The triple second-strip inequality holds at cutoff four.
\end{proposition}
\begin{proof}
On support $m\ge7$, averaging and \eqref{ss3-eq:degree} give the following
bounds.
\[
\begin{array}{c|rrrrrr}
 N&11&12&13&15&16&18\\ \hline
 T_4(N-10)&13&16&17&20&21&22\\
 \delta\text{ is at most}&4&5&5&6&6&6\\
 |\mathcal B|\text{ is at least}&7&7&8&9&10&12
\end{array}
\]
For $N=11$, the bound four uses $\delta\le\lfloor33/7\rfloor$.
The other entries follow directly from \eqref{ss3-eq:degree}.

If $\mathcal B$ has support at least six, choose, respectively,
$j=7,7,8,9,10,10$ surviving triples with that property.
Such a choice is always possible: at most four triples are needed
to cover at least six points, and the selection can then be enlarged.
Lemma~\ref{ss3-lem:gap} gives at least $j+(j+2)=2j+2$ nonempty unions.
Corollary~\ref{cor:ind-bridge} therefore gives
$17,17,19,21,23,23$ large unions of $\mathcal A$, enough in every case.

If the surviving support $Z$ has five points, then $m=7$ and two
points disappeared together. The removed triples are $\{p,y,z\}$
for $\delta$ distinct $z\in Z$, with $2\le\delta\le5$.
Writing $b=N-\delta$, we have $6\le b\le10$ and $N\le15$.
There are at least five large unions avoiding the pair by
\eqref{ss3-eq:first-small}. Through the pair count the
$\binom\delta2$ pairs of the indicated points $z$, the $b$ distinct sets
$B\cup\{z_0\}$ with $B\in\mathcal B$ and a fixed such point $z_0$,
and the full set $Z$.
Their sizes outside the pair are two, three or four, and five,
respectively, so the three counts are disjoint. The resulting bound is
\[
 b+\binom\delta2+6
 =N+\frac{\delta(\delta-3)}2+6\ge N+5,
\]
which covers each of the possible jumps $N=11,12,13,15$.
\end{proof}

We need a little more than the sharp unrestricted bound at three
values of $N$. Excess support supplies it.

\begin{lemma}\label{ss3-lem:three-gaps}
On support at least seven, families of $15$, $16$, and $18$ triples
have at least $21$, $23$, and $29$ unions of size at least four,
respectively.
\end{lemma}
\begin{proof}
If a point has degree one, let $E$ be its sole triple. After its
deletion, the distinct surviving traces outside $E$ have at most two
preimages each. Adjoining $E$ therefore gives at least
$\lceil(N-1)/2\rceil$ further unions of size at least four.
Proposition~\ref{ss3-prop:four} applied to the surviving triples gives
totals at least $24,28,30$.

If minimum degree is at least $m$, averaging permits only $N=18,m=7$.
Lemma~\ref{ss3-lem:degree} gives twenty-two unions of size at least five.
Let $g$ count the generated four-sets. The incidence count
$72\le35+3g$ gives at least
thirteen generated four-sets, more than we need.

In all other cases, delete a minimum-degree point.
The bound \eqref{ss3-eq:degree} gives degree at most six, so at least
$9,10,12$ triples survive. If their support is at least six,
Lemma~\ref{ss3-lem:gap} gives total closure sizes at least $20,22$
in the first two cases. In the third, Proposition~\ref{ss3-prop:four}
gives at least $12+16=28$. Add one using
Corollary~\ref{cor:ind-bridge}. A surviving support of size five is
possible only for $N=15,m=7$: the surviving ten triples form the
complete layer, and the five removed triples share a pair. Those five
triples already give twenty-six large unions; the complete
layer gives six more.
\end{proof}

Whenever one of the preceding bounds is used for more generators,
we can retain support at least seven in the selected subfamily:
at most five triples cover at least seven points.

\begin{proposition}\label{ss3-prop:five}
The triple second-strip inequality holds at cutoff five.
\end{proposition}
\begin{proof}
On support $m\ge8$, the relevant numbers are
\[
\begin{array}{c|rrrrrr}
 N&21&22&23&26&27&30\\ \hline
 T_5(N-20)&18&22&23&27&28&29\\
 |\mathcal B|\text{ is at least}&14&15&16&18&19&21.
\end{array}
\]
The last row follows from \eqref{ss3-eq:degree}. If $\mathcal B$ has support
at least seven, Proposition~\ref{ss3-prop:four} gives seventeen large
unions at $b=14$, while Lemma~\ref{ss3-lem:three-gaps} gives
$21,23,29$ at $b=15,16,18$. Adding one by
Corollary~\ref{cor:ind-bridge} proves all six bounds.

If its support $Z$ has six points, then $m=8$, two points disappeared
together, and the removed triples have the form $\{p,y,z\}$ with
$\delta\le6$. Thus $b=N-\delta\ge15$ and $N\le26$. At most five
triples on $Z$ are missing, so every five-subset of $Z$ is generated:
failure would require all six triples through an uncovered point in
that five-subset to be missing. Together with $Z$, these give seven
unions of size at least five avoiding the pair.

Fix one of the points $z_0$ occurring with the pair. Through the pair,
the $b$ distinct sets $B\cup\{z_0\}$, $B\in\mathcal B$, give unions
of size five or six. The five
five-subsets of $Z$ containing $z_0$ give five further unions of size
seven; $Z$ gives one of size eight. The total is at least
$b+13\ge28$, enough for every jump with $N\le26$.
\end{proof}

\subsection{Traces at the small cutoffs}

The first deletion lets us lower the cutoff, but its remaining family
may not give the whole target. The second deletion separates unions
avoiding a point from additional unions containing it. Retention bounds
the number of surviving triples; trace capacity turns that number into
a lower bound for the additional unions.

Assume the degree reductions have given $2\le\delta\le m-1$.
For triples the first deletion leaves support-size $w\ge\delta$,
so we use Lemma~\ref{lem:ind-retention} with $b=0$.

Let $p$ be the first point deleted and $\mathcal B$ the surviving family.
Delete a minimum-degree point $x$ of $\mathcal B$, leaving
$\mathcal C$. Choose a triple $E\in\mathcal B$ through $x$. Let $\mathcal D$ be the
union-closure, including the empty set, of the traces $C\setminus E$
for $C\in\mathcal C$. Put
\begin{equation}\label{ss3-eq:trace-def}
 f_j=|\mathcal D_j|,\qquad W=f_1+2f_2+f_3,\qquad
 Q_h=|\mathcal D_{\ge h}|=\sum_{j\ge h}f_j.
\end{equation}
Lemma~\ref{lem:ind-trace}, followed by
Corollary~\ref{cor:ind-bridge} at the first deleted point, gives
\begin{equation}\label{ss3-eq:mass}
 |\mathcal C|\le W,
\end{equation}
and
\begin{equation}\label{ss3-eq:two-delete}
 M_n(\mathcal A)\ge1+M_{n-1}(\mathcal C)+Q_{n-4}.
\end{equation}
At cutoff $n-1$ in $\mathcal B$, the two terms after the initial $1$
count unions avoiding $x$ and unions through $x$, respectively. The
cutoff-raising lemma transfers their sum back to $\mathcal A$ and
supplies the extra union. Thus a lower bound for $Q_{n-4}$ completes
the two-deletion argument. At cutoff six we need the following estimate.

\begin{lemma}\label{ss3-lem:q2}
For any finite union-closed family $\mathcal D$ with notation \eqref{ss3-eq:trace-def},
$W\ge22$ implies $Q_2\ge14$.
\end{lemma}
\begin{proof}
Suppose $Q_2\le13$. Two distinct pairs have a union of size three
or four, with at most three representations as an unordered pair of
pairs. Thus
\[
 \binom{f_2}{2}\le3(f_3+f_4)\le3(13-f_2),
\]
so $f_2\le6$. Also $f_1\le4$, since five singletons generate
twenty-six sets of size at least two. If $f_1\le3$ and $f_2\le5$,
then $W\le f_1+f_2+Q_2\le21$. If $f_2=6$, the union of those
pairs has size at least four, so the last inequality improves by
one, again giving $W\le21$.

If $f_1=4$, their four-point support generates eleven sets of size
at least two. Any pair or triple not contained in that support
would give at least four further such unions by adjoining subsets
of the singleton support. Indeed, it meets that support in at most
two points. This contradicts $Q_2\le13$. All pair and triple
members therefore lie on those four points, and $W=4+12+4=20$.
\end{proof}

\subsection{The cutoff six}

\begin{proposition}\label{ss3-prop:six}
The triple second-strip inequality holds at cutoff six.
\end{proposition}
\begin{proof}
Theorem~\ref{thm:second-minimum} treats support eight. On support $m\ge9$, induct
on support and use Lemmas~\ref{ss3-lem:degree} and~\ref{ss3-lem:one} to
assume $2\le\delta\le m-1$. Here $N=35+R$ and the largest target
is $H(8)=37$.

We distinguish the lower steps, where one deletion almost suffices,
from the upper steps, where we use the trace estimate.

\emph{Lower steps: $R\le6$.} For $R=1,2$, \eqref{ss3-eq:degree} leaves at
least twenty-seven triples. Proposition~\ref{ss3-prop:five} and
Corollary~\ref{cor:ind-bridge} give $28+1=29$, enough for the targets
$24,29$. For $R=5,6$, at least thirty triples survive, giving
$29+1=30$, again enough.

It remains to obtain thirty large unions when $N=38,39$. At least
twenty-eight triples survive, giving twenty-eight unions at cutoff
five. If $m\ge12$, the stronger form of Corollary~\ref{cor:ind-bridge}
adds at least two. If $\delta\le8$, at least thirty triples survive,
which also settles the case. This leaves $m=10,\delta=9$, or
$m=11,\delta\in\{9,10\}$. When $\delta\ge m-1$, all
$(m-1)$-subsets are generated, since a one-point cover of a link
permits at most $m-2$ edges. At least two of these unions contain
$p$ and have size greater than eight, so they supply the required
extra two unions.

For $m=11,\delta=9$, the same conclusion holds unless some
ten-subset is not generated. In that event an uncovered point has
all link edges through the omitted point. Minimum degree nine
forces this link to be the full nine-edge star. The resulting nine
triples on a fixed pair generate $\binom94=126$ distinct six-sets,
which is more than enough.

\emph{Upper steps: $R\ge7$.} If the support after the first deletion is seven,
then two points disappeared together. Capacity gives
$R\le\delta\le7$, so $R=\delta=7$: the remaining triples form
the complete layer on a seven-set $Z$, and the removed triples
are $\{p,y,z\}$ for all $z\in Z$. The complete layer gives eight
unions of size at least six; the pair and the subsets of $Z$ of
size at least five give twenty-nine more.

Otherwise the surviving support is at least eight. Since $N\ge42$,
Lemma~\ref{lem:ind-retention} with $a=8$, $b=0$, leaves at least
$\lceil34\cdot5/8\rceil=22$ triples after two deletions.
Proposition~\ref{ss3-prop:five} gives
$M_{5}(\mathcal C)\ge22$. Equations~\eqref{ss3-eq:mass} and
\eqref{ss3-eq:two-delete}, with Lemma~\ref{ss3-lem:q2}, now give
$M_{6}(\mathcal A)\ge1+22+14=37$.
\end{proof}

\subsection{Every larger cutoff}

We first explain why the trace estimates become easier under induction.
Each uniform level of a union-closed family generates a subfamily of
that same family. Thus a bound at a smaller cutoff limits the number
of triples in a family with few large members. The theorem for pairs
does the same for its pair members.

\begin{lemma}\label{ss3-lem:trace-general}
Let $n\ge7$, and assume Theorem~\ref{thm:small-k}(i) at all cutoffs
smaller than $n$. Define
\begin{equation}\label{ss3-eq:tn}
 s_n=\left\lceil\frac{(n-1)(n-2)}{2(n+2)}\right\rceil,
 \qquad t_n=\binom n3+s_n.
\end{equation}
For any finite union-closed family $\mathcal D$ with notation \eqref{ss3-eq:trace-def},
$W\ge t_n$ implies
\[
 Q_{n-4}\ge
 \begin{cases}16,&n=7,\\2n+1,&n\ge8.\end{cases}
\]
\end{lemma}
\begin{proof}
For $n=7,8,9$, suppose respectively that
$Q_3\le15$, $Q_4\le16$, or $Q_5\le18$. The resulting level bounds
are
\[
\begin{array}{c|rrr|r|r}
 n&f_1\le&f_2\le&f_3\le&W\le&t_n\\ \hline
 7&4&8&9&29&37\\
 8&5&10&12&37&59\\
 9&6&15&21&57&87.
\end{array}
\]
Each column follows from a lower-cutoff bound. Five,
six, and seven singletons generate respectively $16,22,29$ unions
at cutoffs $3,4,5$. Formula~\eqref{ss3-eq:pairs} gives at least
$16,17,23$ such unions from $9,11,16$ pairs. Finally ten triples
give at least sixteen nonempty unions by \eqref{ss3-eq:first-small};
thirteen triples give seventeen unions at cutoff four by
Proposition~\ref{ss3-prop:four}; and twenty-two triples give twenty-two
at cutoff five by Proposition~\ref{ss3-prop:five}. Each row contradicts
$W\ge t_n$.

Suppose now $n\ge10$, put $h=n-4$, and assume $Q_h\le2n$.
The following thresholds each force more than $2n$ unions of size
at least $h$:
\[
\begin{array}{c|c|c}
 \text{member size}&\text{number of members}&
 \text{forced unions of size at least }h\\ \hline
 1&n-2&H(n-2)\\
 2&\binom{n-3}{2}+1&(n^2-5n+10)/2\\
 3&\binom{n-3}{3}+1&(n^2-7n+18)/2.
\end{array}
\]
The first row counts unions of singletons. The second is
\eqref{ss3-eq:pairs} at the first value of its second strip. The third
is the induction hypothesis at cutoff $h$, using
\eqref{ss3-eq:lower-steps}. Consequently $f_j\le\binom{n-3}{j}$
for $j=1,2,3$. The Vandermonde identity in
Lemma~\ref{lem:ind-trace} now gives
\[
 W\le\binom{n-1}{3}<t_n,
\]
a contradiction.
\end{proof}

\begin{proof}[Completion of Theorem~\ref{thm:small-k}(i)]
The cutoffs $3,4,5,6$ have been proved. Fix $n\ge7$ and assume
the theorem at all smaller cutoffs. At cutoff $n$, induct on the
support. The case $m=n+2$ is Theorem~\ref{thm:second-minimum}. For $m\ge n+3$, the degree
lemmas again reduce us to $2\le\delta\le m-1$.

For $R\le n$ the first deletion suffices. For $R\ge n+1$ we use
the second deletion and the trace estimate, except when the first
deletion leaves a complete layer.

\emph{Lower steps: $R\le n$.} The quantity $N-d(N)$ is nondecreasing
in $N$: the integer $d(N)$ increases by at most one when $N$
increases by one. For $n\ge8$, put $N_0=\binom{n+1}{3}+1$.
Then
\begin{equation}\label{ss3-eq:low-R}
 d(N_0)\le\binom{n-2}{2},\qquad
 N_0-d(N_0)\ge\binom n3+2n-2.
\end{equation}
For the first inequality, put $b=\binom{n-2}{2}$ and compare
$(b+1)(b+2)$ with $3N_0$. The difference is $17$ at $n=8$,
and its expansion in $n-8$ has positive coefficients. The second
inequality follows from Pascal's identity. Thus the first deletion leaves enough
triples to reach the last step of the second strip at cutoff $n-1$.
The induction hypothesis gives at least $H(n+1)$ unions there.
Corollary~\ref{cor:ind-bridge} adds one, which proves all the required
bounds in \eqref{ss3-eq:lower-steps}.

For $n=7$, the three relevant starts are $N=57,58,59$. The degree
bound leaves at least $45,46,47$ triples. At cutoff six these give
$36,36,37$ unions. Adding one proves the targets $31,37,38$.
Monotonicity covers the intervals between these starts.

\emph{Upper steps: $R\ge n+1$.} If the first deletion leaves
support $n+1$, then $m=n+3$ and two points disappeared together.
Capacity gives $R\le\delta\le n+1$, so $R=\delta=n+1$.
The surviving family is the complete triple layer on an $(n+1)$-set
$Z$, and the removed triples are $\{p,y,z\}$ for all $z\in Z$.
Unions avoiding the pair give $n+2$ large sets. Through the pair,
the subsets of $Z$ of size at least $n-1$ give $H(n+1)$ more.
Their sum is $H(n+2)$.

Otherwise the surviving support is at least $n+2$.
Lemma~\ref{lem:ind-retention}, with $a=n+2$, $b=0$, and
$N\ge\binom{n+1}{3}+n+1$, gives at least $t_n$ triples after the
second deletion, with $t_n$ as in \eqref{ss3-eq:tn}. Notice that
$s_7=2$ and $3\le s_n\le n-1$ for $n\ge8$.
The induction hypothesis at cutoff $n-1$ therefore gives
\[
 M_{n-1}(\mathcal C)\ge
 \begin{cases}29,&n=7,\\H(n)+1,&n\ge8.\end{cases}
\]
Together with \eqref{ss3-eq:mass}, Lemma~\ref{ss3-lem:trace-general}, and
\eqref{ss3-eq:two-delete}, this yields
\[
 M_{n}(\mathcal A)\ge
 \begin{cases}
 1+29+16=H(9),&n=7,\\
 1+[H(n)+1]+(2n+1)=H(n+2),&n\ge8.
 \end{cases}
\]
This is at least $T_n(R)$ and completes both inductions.
\end{proof}

\subsection{A common bound for small generators}

For $k=3,4,5$, we use support reduction directly. We shall find a core
with at most five points that can become uncovered when three points
are omitted. Cover these exceptions and the points lost in the last
deletion by at most six generators. Any three other points may then
be omitted.

\begin{proposition}\label{prop:second-short-range}
For $k\in\{3,4,5\}$, the second-strip inequality holds on arbitrary
support whenever $n\ge10k$.
\end{proposition}
\begin{proof}
The minimum support is already covered by Theorem~\ref{thm:second-minimum}.
Suppose the support is $n+c$, with $c\ge3$.

\emph{Large supports.} If $c\ge4k$, apply
Lemma~\ref{lem:large-support-cover} with $J=4$. Put
$j=\lceil(n+4k)/k\rceil\ge14$. The lemma gives at least $S_4(j)$
large unions. Since $n\le k(j-4)\le5(j-4)$,
it suffices to compare this with $T_{\max}(5j-20)$. The difference
$S_4(j)-T_{\max}(5j-20)$ is a polynomial in $j-14$ with positive
coefficients, so it is positive for $j\ge14$.

\emph{A dense remaining family.} Suppose $3\le c<4k$.
Delete minimum-degree points until
the actual support has size at most $n+2$, unless the degree
observation has already finished the proof. Each support size is
visited at most once. Summing the deletion costs gives a core with at
least
\[
 K=\binom{n+2}{k}-D+1,
 \qquad
 D=\binom{n+4k-2}{k-1}+\binom{n+4k-4}{k-2}
                        -\binom{n-1}{k-2}
\]
generators. In this expression, $D$ is the sum of $\lambda(w,k)$ over
$n+3\le w\le n+4k-1$, together with $\binom{n+1}{k-1}$; the latter
term changes the reference capacity from $n+1$ points to $n+2$.
Write $\mathcal G$ for this core and $s$ for its support-size.

\emph{Bounding the exceptions.} For a core on $s$ points, call a
point exceptional at level $u$ if
some $u$-set containing it contains no generator through it. If
$x_1,\ldots,x_r$ are exceptional, count the missing generators in
the witness for $x_i$ which avoid $x_1,\ldots,x_{i-1}$. These
collections are disjoint and contain at least
\[
 \sum_{i=1}^r\binom{u-i}{k-1}
   =\binom uk-\binom{u-r}{k}
\]
missing generators. At $u=s-3$, six exceptional points would therefore
give
\[
 |\mathcal G|\le
 F(s):=\binom sk-\binom{s-3}{k}+\binom{s-9}{k}.
\]
The function $F$ is nondecreasing: the difference of its first two
terms is a sum of three nondecreasing binomial coefficients.
Consequently every possible core on $s\le n+2$ has at most five
exceptions if
\begin{equation}\label{eq:second-short-budget}
 D\le\binom{n-1}{k}-\binom{n-7}{k}.
\end{equation}
Indeed, this is exactly $K>F(n+2)$. It also implies
$K>\binom{k+2}{k}$, so the core is nonempty and $s\ge k+3$,
as required for the choice $u=s-3$.

To check \eqref{eq:second-short-budget}, divide by
$\binom{n-2}{k-1}$. Each term in the sum defining $D$ then decreases
with $n$, as follows from the product formula for binomial coefficients.
On the right, write the difference as the sum of six consecutive
binomial coefficients of order $k-1$. Each normalized term increases
with $n$. Thus only $n=10k$ remains, and substitution verifies the
inequality for $k=3,4,5$. The right side exceeds the left by
$1094$, $15942$, and $102565$, respectively.

\emph{Counting the omissions.} Let $Z$ be the support just before the
final deletion. The points
lost in that step have identical incidence families, so one deleted
generator contains them all. Choose that generator and one core
generator through each exceptional point. Their union $W$ has at most
$6k$ points. Every set $Z\setminus D'$ with
$D'\subseteq Z\setminus W$, $|D'|\le3$, is generated. Indeed, the
chosen generators cover $W$. Every other point lies in the remaining
core, whose size is at least $s-3$, and is covered there by
nonexceptionality. Since $|Z|\ge n+3$, these generated sets have size
at least $n$ and number at least $S_3(n+3-6k)$.

Finally, $n\ge10k$ gives $n+3-6k\ge2n/5+3$. Evaluating $S_3$ by
its polynomial expression, the difference
$S_3(2n/5+3)-T_{\max}(n)$ is a cubic in $n-30$ with positive
coefficients. It is therefore positive for $n\ge30$. This finishes
the remaining support sizes and proves the proposition.
\end{proof}

The restriction $k\in\{3,4,5\}$ is used in the cover comparison
and in checking the core budget. Taking $k=4,5$ gives $n\ge40,50$, respectively, and proves
Theorem~\ref{thm:small-k}(ii). Equality in both parts is supplied by
the max-lex family, whose count was computed in
subsection~\ref{the-max-lex-count}.

\appendix
\section{Third-strip calculations for triples}\label{app:third-strip}

We complete the calculations for Theorem~\ref{thm:third-small}, using
the deletion estimate of Proposition~\ref{prop:third-deletion}.
Throughout this appendix, $k=3$ and $n\ge3$ is an integer. Let
$\mathcal A$ be a family of $N$ triples with support $V$ of size $n+3$,
where $N=\binom{n+2}3+R$ and $1\le R\le\binom{n+2}2$.
Put $v=n+2$ and $q=n-k+3$. As in Definition~\ref{def:small-intersections},
$h_j^{(q)}(r)$ counts the $j$-subsets contained in the first $r$
$q$-sets in colex order. Write
\[
 Q_q(R)=\sum_{j=0}^3h_j^{(q)}(R),\qquad
 B(n,k)=\binom n{k-1}+\binom{n-2}{k-1}+\binom{n-3}{k-1}.
\]
Thus $Q_q$ is the function $Q$ in \eqref{eq:third-arithmetic}. We also write
$S_2(v)=1+v+\binom v2$. In the deletion argument,
$\mathcal B$ is the family avoiding the deleted point and $\mathcal G$ consists of the
remaining generators with that point removed. Write $z$ for the deleted
point and $X=V\setminus\{z\}$. Its degree $d=|\mathcal G|$ satisfies
\[
 R\le d\le D(n,k)=\binom{n+2}{k-1}-\binom n{k-1},\qquad e=d-R.
\]
For $e>0$, let $y\ge k-1$ satisfy $e=\binom y{k-1}$ and, for integers
$s\ge0$, set
\[
 b_s(y)=\begin{cases}\lfloor\binom ys\rfloor,&y\ge s,\\0,&y<s.\end{cases}
\]
Proposition~\ref{prop:continuous-missing} bounds the number $c_u$ of missing $u$-unions of $\mathcal B$
by $b_{u-1}(y)$. Put $L_{n,k}(0)=0$ and
\begin{equation}\label{app:eq:loss}
 L_{n,k}(e)=b_{n-2}(y)+2\bigl(b_{n-1}(y)+b_n(y)+b_{n+1}(y)\bigr)
 \quad(e>0).
\end{equation}
The deletion bound, Proposition~\ref{prop:third-deletion}, is
\begin{equation}\label{app:eq:actual-loss}
 |\langle\mathcal A\rangle_{\ge n}|\ge S_2(v)+Q_q(d)-c_{n-1}
                         -2(c_n+c_{n+1}+c_{n+2}).
\end{equation}
Consequently it suffices to check
\begin{equation}\label{app:eq:criterion}
 Q_q(d)-Q_q(R)\ge L_{n,k}(d-R).
\end{equation}
If the degree interval is empty, Lemma~\ref{lem:bottom-level} applies.

Suppose $n>3$.
For $k=3$, the difference between $B(n,3)$ and the largest possible
missing-set budget plus one is
\[
 B(n,3)-\binom{n+2}2=n^2-8n+8.
\]
It is positive for every $n\ge7$. Theorem~\ref{thm:h-criterion} therefore leaves
only the following ranges:
\begin{center}
\begin{tabular}{ccl}
\toprule
$n$ & Remainders still to prove & Degree bound \\
\midrule
4 & $1\le R\le8$ & $d\le9$ \\
5 & $1\le R\le7$ & $d\le11$ \\
6 & $1\le R\le4$ & $d\le13$ \\
\bottomrule
\end{tabular}
\end{center}
In each case, Lemma~\ref{lem:bottom-level} permits us to assume that a blocked
$(n+1)$-set exists. Delete a blocking point as in
Proposition~\ref{prop:third-deletion}, and put $e=d-R$.

\subsection{Cutoff six}
Here $q=6$ and $0\le e\le12$. If $e\le5$, a missing $5$-set would
require a star of six missing triples. Thus $L_{6,3}(e)=0$, and
monotonicity of $Q_6$ proves~\eqref{app:eq:criterion}.

If $e\ge6$, the defining real number $y$ satisfies $y<11/2$, since
$\binom{11/2}2=99/8>12$. Therefore
\[
 L_{6,3}(e)\le
 \left\lfloor\binom{11/2}4\right\rfloor
 +2\left\lfloor\binom{11/2}5\right\rfloor=9+2\cdot2=13.
\]
For $R=1,2,3,4$, respectively,
\[
 Q_6(R)=42,58,63,64,
 \qquad Q_6(R+6)=64,80,85,86.
\]
Since $d\ge R+6$, the gain is at least $22$ in every case.
This proves the remaining cutoff-six cases.

\subsection{Cutoff five}
Here $q=5$ and $0\le e\le10$. Again $e\le2$ gives
$L_{5,3}(e)=0$. The remaining comparisons are short enough to record
in full. The last column is the minimum over
$1\le R\le\min\{7,11-e\}$.
\begin{center}
\begin{tabular}{ccc}
\toprule
$e$ & Upper bound for $L_{5,3}(e)$ &
$\min\bigl(Q_5(R+e)-Q_5(R)\bigr)$ \\
\midrule
3 & 1 & 1 \\
4 & 1 & 5 \\
$5\le e\le8$ & 11 & $\ge16$ \\
9 & 14 & 21 \\
10 & 22 & 32 \\
\bottomrule
\end{tabular}
\end{center}
For the middle column, use $y=3$ at $e=3$, $y<17/5$ at $e=4$,
$y<23/5$ for $e\le8$, $y<24/5$ at $e=9$, and $y=5$ at $e=10$.
These inequalities follow by substituting in $\binom y2$; substituting
in~\eqref{app:eq:loss} gives the displayed bounds. The last column follows
from the complete list
\[
 \bigl(Q_5(1),\ldots,Q_5(11)\bigr)
 =(26,37,41,42,42,42,53,57,58,58,58).
\]
Thus~\eqref{app:eq:criterion} holds in every remaining cutoff-five case.

\subsection{Cutoff four}
Here $X$ has six points and $\mathcal G$ is a graph whose edges meet two
fixed points $a,b$. There are nine possible edges: $ab$ and the eight
edges from $a,b$ to the other four points.

We first apply the numerical count wherever it suffices. Its inputs are
\begin{center}
\begin{tabular}{c|rrrrrrrrr}
\toprule
$d$ &1&2&3&4&5&6&7&8&9\\
$Q_4(d)$ &15&22&25&26&26&33&36&37&37\\
\midrule
$e$ &0&1&2&3&4&5&6&7&8\\
$L_{4,3}(e)$ &0&1&2&5&6&9&16&19&24\\
\bottomrule
\end{tabular}
\end{center}
The second row is the colex shadow count. For the last row use
$y=(1+\sqrt{1+8e})/2$ in~\eqref{app:eq:loss}; its first term is exactly $e$.
Comparison of the rows proves~\eqref{app:eq:criterion}, except possibly at
\begin{equation}\label{app:eq:exceptions}
 \begin{split}
 d=5 &: \quad R=2,3,4,\\
 d=8 &: \quad R=2,\\
 d=9 &: \quad R=1,2,3,6,7,8.
 \end{split}
\end{equation}
We treat these cases by counting some unions through $z$ directly.

Suppose first $d=5$. If all edges of $\mathcal G$ share a point, $\mathcal G$ is
the full star on six points. It generates every set of size at least
three containing its centre; there are $26$ of these. After adjoining
$z$, they give $26$ large unions, irrespective of the missing triples
in $\mathcal B$. Since $e\le3$, Proposition~\ref{prop:continuous-missing} shows that at most
one subset of $X$ of size at least four is missing as a union, and none
is missing when $e\le2$. The total is at least $47$ when $R=2$ and
at least $48$ when $R=3,4$. The required values are $44,47,48$.

If the five edges have no common point, their complements are five
$4$-sets whose union is all six points. Their singleton shadow thus
has size six. Kruskal--Katona gives at least ten pairs and ten triples
in their shadows. Consequently the candidate count used in
\eqref{app:eq:actual-loss} is at least $1+6+10+10=27$, one more than
$Q_4(5)$. This extra candidate supplies exactly the shortfall in each
of the three $d=5$ cases in~\eqref{app:eq:exceptions}.

Next let $d=8$ and $R=2$, so $e=6$. If the missing edge is $ab$,
the graph is $K_{2,4}$. All $37$ subsets of size at least three meeting
both parts are generated by its edges. Otherwise, after exchanging
$a,b$ if necessary, the missing edge is $ac$ for one of the other
four points. The point $b$ is adjacent to every other point.
The $26$ subsets of size at least three containing $b$ are generated.
There are four further generated subsets avoiding $b,c$: take $a$
and at least two of the other three points. Thus there are at least
$30$ unions through $z$. Since $e=\binom42$, Proposition~\ref{prop:continuous-missing}
bounds the missing old unions by $\binom43+\binom44=5$.
The total is at least $30+(22-5)=47$, exceeding the required $44$.

Finally let $d=9$. Every subset of $X$ of size at least three meeting
$\{a,b\}$ is generated by $\mathcal G$. Exactly $37$ subsets have this property.
If $R\le5$, then $e\le8$ and $y<23/5$. There are at most
\[
 \left\lfloor\binom{23/5}3\right\rfloor
 +\left\lfloor\binom{23/5}4\right\rfloor=7+2=9
\]
missing old unions of size at least four. Hence the total is at least
$37+22-9=50$, while the target is at most $48$.
If $R=6$, then $e=3$, and the total is at least $37+21=58>55$.
For $R=7,8$, we have $e\le2$, so all $22$ old large subsets are
generated. The total $59$ is at least the respective targets $58,59$.
This completes the cutoff-four proof.

\subsection{The diagonal cutoff}
When $n=3$, the same range of $N$ is the second strip for cutoff four,
on its minimum support of six points. Apply the established
minimum-support second-strip theorem and add the $N$ triples:
\[
 |\langle\mathcal A\rangle_{\ge3}|=N+|\langle\mathcal A\rangle_{\ge4}|
 \ge N+U_3(N,4)=U_3(N,3).
\]
Together with the preceding cases, this proves Theorem~\ref{thm:third-small}.

\hypertarget{acknowledgments}{%
\section*{Acknowledgments}}

The author thanks Farhood Rostamkhani for his contributions to the early development of this work and Žarko Ranđelović for insightful communication.

\Needspace{28\baselineskip}
\bibliographystyle{amsplain}
\bibliography{references}

@phdthesis{Roberts1999,
  author  = {Roberts, Ian T.},
  title   = {Extremal Problems and Designs on Finite Sets},
  school  = {Curtin University},
  address = {Perth, Australia},
  year    = {1999}
}

@article{LeckRobertsSimpson2012,
  author  = {Leck, Uwe and Roberts, Ian T. and Simpson, Jamie},
  title   = {Minimizing the weight of the union-closure of families of two-sets},
  journal = {Australasian Journal of Combinatorics},
  volume  = {52},
  year    = {2012},
  pages   = {67--73},
  note    = {\href{https://ajc.maths.uq.edu.au/pdf/52/ajc_v52_p067.pdf}{Full text}}
}

@article{Randjelovic2026,
  author  = {Ran{\dj}elovi{\'c}, {\v Z}arko},
  title   = {Minimizing the number of unions},
  journal = {Electronic Journal of Combinatorics},
  volume  = {33},
  number  = {1},
  year    = {2026},
  pages   = {Paper No. P1.47, 15 pp.},
  note    = {\href{https://doi.org/10.37236/13702}{doi:10.37236/13702}}
}

@misc{IMC2016,
  author       = {{International Mathematics Competition for University Students}},
  title        = {Day 1, Problem 4},
  year         = {2016},
  howpublished = {\href{https://www.imc-math.org.uk/?item=prob4q&section=problems&year=2016}{Official problem and solution}},
  note         = {Proposed by Fedor Petrov}
}

@article{Lubell1966,
  author  = {Lubell, D.},
  title   = {A short proof of {Sperner}'s lemma},
  journal = {Journal of Combinatorial Theory},
  volume  = {1},
  year    = {1966},
  pages   = {299},
  note    = {\href{https://www.sciencedirect.com/science/article/pii/S0021980066800352}{Publisher's page}}
}

@article{BjornerKalai1988,
  author  = {Bj{\"o}rner, Anders and Kalai, Gil},
  title   = {An extended {Euler--Poincar{\'e}} theorem},
  journal = {Acta Mathematica},
  volume  = {161},
  year    = {1988},
  pages   = {279--303},
  note    = {\href{https://doi.org/10.1007/BF02392300}{doi:10.1007/BF02392300}}
}

@misc{ZhanHuang2025,
  author       = {Zhan, Xiongfeng and Huang, Xueyi},
  title        = {A note on the {Bj{\"o}rner--Kalai} theorem},
  year         = {2025},
  howpublished = {\href{https://arxiv.org/abs/2504.05943v3}{arXiv:2504.05943v3}},
  note         = {Version of 4 December 2025}
}

@article{ChaoYu2024,
  author  = {Chao, Ting-Wei and Yu, Hung-Hsun Hans},
  title   = {{Kruskal--Katona}-type problems via the entropy method},
  journal = {Journal of Combinatorial Theory, Series B},
  volume  = {169},
  year    = {2024},
  pages   = {480--506},
  note    = {\href{https://doi.org/10.1016/j.jctb.2024.08.003}{doi:10.1016/j.jctb.2024.08.003}}
}

@misc{ChaoYu2023Joints,
  author       = {Chao, Ting-Wei and Yu, Hung-Hsun Hans},
  title        = {Tight bound and structural theorem for joints},
  year         = {2023},
  howpublished = {\href{https://arxiv.org/abs/2307.15380v3}{arXiv:2307.15380v3}},
  note         = {Version of 22 December 2023}
}

@mastersthesis{McLeod2002,
  author  = {McLeod, Jeanette},
  title   = {Minimisation Problems and Linear Orders in the {Boolean} Lattice},
  type    = {Honours thesis},
  school  = {Northern Territory University},
  address = {Darwin, Australia},
  year    = {2002}
}

@incollection{Kruskal1963,
  author    = {Kruskal, Joseph B.},
  title     = {The number of simplices in a complex},
  booktitle = {Mathematical Optimization Techniques},
  editor    = {Bellman, Richard},
  publisher = {University of California Press},
  address   = {Berkeley, CA},
  year      = {1963},
  pages     = {251--278}
}

@incollection{Katona1968,
  author    = {Katona, Gyula O. H.},
  title     = {A theorem of finite sets},
  booktitle = {Theory of Graphs (Proc. Colloq., Tihany, 1966)},
  editor    = {Erd{\H o}s, P. and Katona, G.},
  publisher = {Academic Press and Akad{\'e}miai Kiad{\'o}},
  address   = {New York and Budapest},
  year      = {1968},
  pages     = {187--207}
}
\end{document}